\documentclass[12pt]{amsart}
\pdfoutput=1  
\usepackage{graphicx}
\usepackage[T1]{fontenc}
\usepackage{amssymb}
\usepackage[mathscr]{eucal}
\usepackage{amsmath,amssymb,latexsym,bbm}
\usepackage{color}

\newtheorem{thm}{Theorem}[section]
\newtheorem{lemma}[thm]{Lemma}
\newtheorem{prop}[thm]{Proposition}

\theoremstyle{definition}

\theoremstyle{remark}
\newtheorem{remark}[thm]{Remark}

\numberwithin{equation}{section}

\newcommand{\cA}{{\mathcal A}}
\newcommand{\cB}{{\mathcal B}}
\newcommand{\cF}{{\mathcal F}}

\newcommand{\RR}{{\mathbb R}}
\newcommand{\NN}{{\mathbb N}}

\newcommand{\Var}{\operatorname{Var}}

\newcommand{\Exp}{\mathbb E}
\newcommand{\PP}{\mathbb P}
\newcommand{\Cov}{\operatorname{Cov}}
\newcommand{\dd}{\mathrm{d}}
\newcommand{\ee}{\mathrm{e}}
\newcommand{\proofend}{\hfill\mbox{$\Box$}}

\allowdisplaybreaks

\begin{document}

\sloppy

\title[Sample path deviations of process bridges]{Sample path deviations of the Wiener and \\
 the Ornstein-Uhlenbeck process from its bridges}

\author{M\'aty\'as Barczy}
\address{M\'aty\'as Barczy, Faculty of Informatics, University of Debrecen, Pf.12, H-4010 Debrecen, Hungary}
\email{barczy.matyas\@@{}inf.unideb.hu}
\urladdr{http://www.inf.unideb.hu/valseg/dolgozok/barczy/barczy\_angol.html}
\thanks{The first author has been supported by the Hungarian Scientific Research Fund under
 Grant No.\ OTKA T-079128.
This work has been finished while M. Barczy was on a post-doctoral position at the Laboratoire de Probabilit\'es
 et Mod\`{e}les Al\'eatoires, University Pierre-et-Marie Curie, thanks to NKTH-OTKA-EU FP7 (Marie Curie action)
 co-funded 'MOBILITY' Grant No. OMFB-00610/2010.}

\author{Peter Kern}
\address{Peter Kern, Mathematisches Institut, Heinrich-Heine-Universit\"at D\"usseldorf,
Universit\"atsstr.~1, D-40225 D\"usseldorf, Germany}
\email{kern\@@{}math.uni-duesseldorf.de}
\urladdr{http://www.math.uni-duesseldorf.de/$\sim$stoch/Kern/kern.htm}

\date{\today}

\begin{abstract}
We study sample path deviations of the Wiener process from three different representations
  of its bridge: anticipative version, integral representation and space-time transform.
Although these representations of the Wiener bridge are equal in law, their sample path behavior is quite different. Our results nicely demonstrate this fact.
We calculate and compare the expected absolute, quadratic and conditional quadratic path deviations
 of the different representations of the Wiener bridge from the original Wiener process.
It is further shown that the presented qualitative behavior of sample path deviations is not
 restricted only to the Wiener process and its bridges.
Sample path deviations of the Ornstein-Uhlenbeck process from its bridge versions are also considered and
 we give some quantitative answers also in this case.
\end{abstract}

\keywords{Sample path deviation, Brownian bridge, Ornstein-Uhlenbeck bridge, anticipative version, integral representation, space-time transform.}

\subjclass[2010]{Primary 60G17; Secondary 60G15, 60J65.}

\maketitle

\baselineskip=18pt

\section{Introduction}

Let $(W_{t})_{t\geq0}$ be a standard one-dimensional Wiener process on a filtered probability space
  $(\Omega,\cF,(\cF_t)_{t\geq 0},P)$, where the filtration $(\cF_t)_{t\geq 0}$ is the usual
 augmentation of the natural filtration of the Wiener process $W$
 (see, e.g., Karatzas and Shreve \cite[Section 5.2.A]{KarShr}).
We consider the following versions of the Wiener bridge from $a$ to $b$ over the time-interval
 $[0,T]$, where $a,b\in\RR$ (see, e.g., Karatzas and Shreve \cite[Section 5.6.B]{KarShr}):

\noindent\text{\bf 1. Anticipative version }
\[
 \displaystyle W_{t}^{{\rm av}}
       =a+(b-a)\frac{t}{T}+\left(W_{t}-\frac{t}{T}\,W_{T}\right),\qquad 0\leq t\leq T.
\]
\text{\bf 2. Integral representation }
\[
 W_{t}^{{\rm ir}}
   =
     \begin{cases}
      \displaystyle a+(b-a)\frac{t}{T}+\int_{0}^t\frac{T-t}{T-s}\,\dd W_{s} & \text{if \ $0\leq t<T$,}\\
        b & \text{if \ $t=T$}.
      \end{cases}
\]
\text{\bf 3. Space-time transform }
\[
  W_{t}^{{\rm st}}
  =
   \begin{cases}
     \displaystyle a+(b-a)\frac{t}{T}+\frac{T-t}{T}\,W_{\frac{tT}{T-t}} & \text{ if \ $0\leq t<T$},\\
         b & \text{ if \ $t=T$.}
    \end{cases}
\]
Here the attribute anticipative indicates that for the definition of
 \ $W_{t}^{{\rm av}}$ \ we use the random variable \ $W_T$, \ where the time point
  \ $T$ \ follows the time point \ $t$.
\ In the sequel we will use the notation $(W_{t}^{\rm br})_{t\in[0,T]}$ if the version of
 the bridge is not specified.
All the bridge versions above are Gauss processes with the same finite-dimensional distributions.
This can be easily calculated, since the versions all have mean function $\Exp(W_t^{{\rm br}})=a+(b-a)\frac{t}{T}$,
 $0\leq t\leq T$ and for $0\leq s\leq t<T$ we have the covariance function
\begin{align*}
\Cov(W_{s}^{\rm av},W_{t}^{\rm av}) & =\Cov\left(W_s-\frac{s}{T}W_T,W_t-\frac{t}{T}W_T\right)\\
& = s-\frac{st}{T}-\frac{st}{T}+\frac{st}{T}=s\frac{T-t}{T},\\
    \Cov(W_{s}^{\rm ir},W_{t}^{\rm ir}) &
 =\Cov\left(\int_0^s\frac{T-s}{T-r}\,\dd W_r,\int_0^t\frac{T-t}{T-r}\,\dd W_r\right)\\
& =\int_0^s\frac{(T-s)(T-t)}{(T-r)^2}\,\dd r=(T-s)(T-t)\left(\frac1{T-s}-\frac1T\right)=s\frac{T-t}{T}
\intertext{and}
\Cov(W_{s}^{\rm st},W_{t}^{\rm st}) &
  =\Cov\left(\frac{T-s}{T}W_{\frac{sT}{T-s}},\frac{T-t}{T}W_{\frac{tT}{T-t}}\right)\\
& = \frac{(T-s)(T-t)}{T^2}\cdot\frac{sT}{T-s}=s\frac{T-t}{T},
\end{align*}
 where we used that the function $[0,T)\ni t\mapsto \frac{tT}{T-t}$ is monotone increasing.
Altogether, for all $0\leq s\leq t<T$ we have
\begin{equation}\label{bridgecov}
\Cov(W_{s}^{\rm br},W_{t}^{\rm br})=s\frac{T-t}{T}.
\end{equation}

We note that the finite dimensional distributions of the above Wiener bridge versions coincide
 with the conditional finite dimensional distributions of the Wiener process $(a+W_{t})_{t\in[0,T]}$ starting
 in $a$ and conditioned on $\{W_{T}=b\}$; see, e.g., Problem 5.6.13 in Karatzas and Shreve
 \cite{KarShr} or Chapter IV.4 in Borodin and Salminen \cite{BorSal}.
Bridges of Gaussian processes have been generally defined by Gasbarra et al. \cite{GasSotVal}, while from the Markovian
 point of view the reader may consult Fitzsimmons et al. \cite{FitPitYor}, Barczy and Pap \cite{BarPap2},
 Chaumont and Uribe Bravo \cite{ChaBra}, and the more recent Bryc and Weso\l owski
 \cite{BryWes} which deals with the inhomogeneous case.

 Moreover, it follows from the definitions that all bridge versions have almost sure continuous sample paths.
 The (left) continuity of the trajectories at $t=T$ is not obvious in case of the integral representation
 and space-time transform.
Corollary 5.6.10 in Karatzas and Shreve \cite{KarShr} yields the desired continuity for the integral representation,
 whereas the strong law of large numbers for a standard Wiener process (see, e.g., Problem 2.9.3
 in Karatzas and Shreve \cite{KarShr}) for the space-time transform.
Hence the anticipative version \ $W^{\rm av}$, \ the integral representation \ $W^{\rm ir}$ \ and
 the space-time transform \ $W^{\rm st}$ \ induce the same probability measure on
 \ $(C[0,T],\cB(C[0,T]))$, \ where \ $C[0,T]$ \ is the space of continuous functions from \ $[0,T]$
 \ into \ $\RR$ \ and \ $\cB(C[0,T])$ \ denotes the Borel $\sigma$-algebra on \ $C[0,T]$.
\ This underlines and explains the definition of a Wiener bridge from \ $a$ \ to \ $b$ \ over the time-interval
\ $[0,T]$ \ (see, e.g., Karatzas and Shreve \cite[Definition 5.6.12]{KarShr}), namely,
 it is any almost surely continuous Gauss process having mean function
 \ $a+(b-a)\frac{t}{T}$, $t\in[0,T]$, \ and covariance function given in \eqref{bridgecov}.

Furthermore, according to Section 5.6.B in Karatzas and Shreve \cite{KarShr}
 or Example 8.5 in Chapter IV in Ikeda and Watanabe \cite{IkeWat},
 the above versions of the Wiener bridge are solutions to the linear
 stochastic differential equation (SDE)
\begin{equation}\label{bridgesde}
 \dd W_{t}^{\rm br}
  = \frac{b-W_{t}^{\rm br}}{T-t}\,\dd t + \dd W_{t}\,,\quad 0\leq t<T\,,
     \quad\text{ with }\quad W_{0}^{\rm br}=a.
\end{equation}
By Theorem 5.2.1 in {\O}ksendal \cite{Oks} or Theorem 2.32 in Chapter III
 in Jacod and Shiryaev \cite{Shi}, strong uniqueness holds for the SDE \eqref{bridgesde},
 and \ $(W_{t}^{\rm ir})_{t\in[0,T)}$ \ is the unique strong solution of this SDE
 being adapted to the filtration $(\mathcal F_{t})_{t\in[0,T)}$.
Whereas $(W_{t}^{\rm av})_{t\in[0,T)}$ is only a weak solution to the SDE \eqref{bridgesde};
 it can not be a strong solution, since the definition of the anticipative representation
 formally requires information about $W_{T}$, although $W_{t}^{\rm av}$ and $W_{T}$ are independent
 for every $t\in[0,T]$ (indeed,
  $\Cov(W_{t}^{\rm av},W_{T})=\Cov(W_{t},W_{T})-\frac{t}{T}\,\Cov(W_{T},W_{T})=0$, $t\in[0,T]$).
The space-time transform representation $(W_{t}^{\rm st})_{t\in[0,T)}$ is only a weak solution
 to the SDE \eqref{bridgesde}, too, since it is adapted
 only to the filtration $(\mathcal F_{\frac{tT}{T-t}})_{t\in[0,T)}$ and
 $\mathcal F_{\frac{tT}{T-t}} \supsetneqq \mathcal F_{t}$, $t\in(0,T)$.
We also note that, even though the three bridge versions have the same law on $(C[0,T],\cB(C[0,T]))$,
 their joint laws together with the Wiener process through which they are constructed, are different
 (see Propositions \ref{correlation} and \ref{prop2}).
Our aim is to elucidate
 their sample path deviations compared to the original Wiener process $(a+W_{t})_{t\in[0,T]}$ starting in $a$.
 A motivation for our study is given at the end of this section.

According to simulation studies, for a typical sample path of the Wiener process the deviations
 from its anticipative bridge version and its space-time transform are larger than from
 its integral representation of the bridge; see Figure \ref{Figure1}.
\begin{figure}[h]
\includegraphics[scale=0.25]{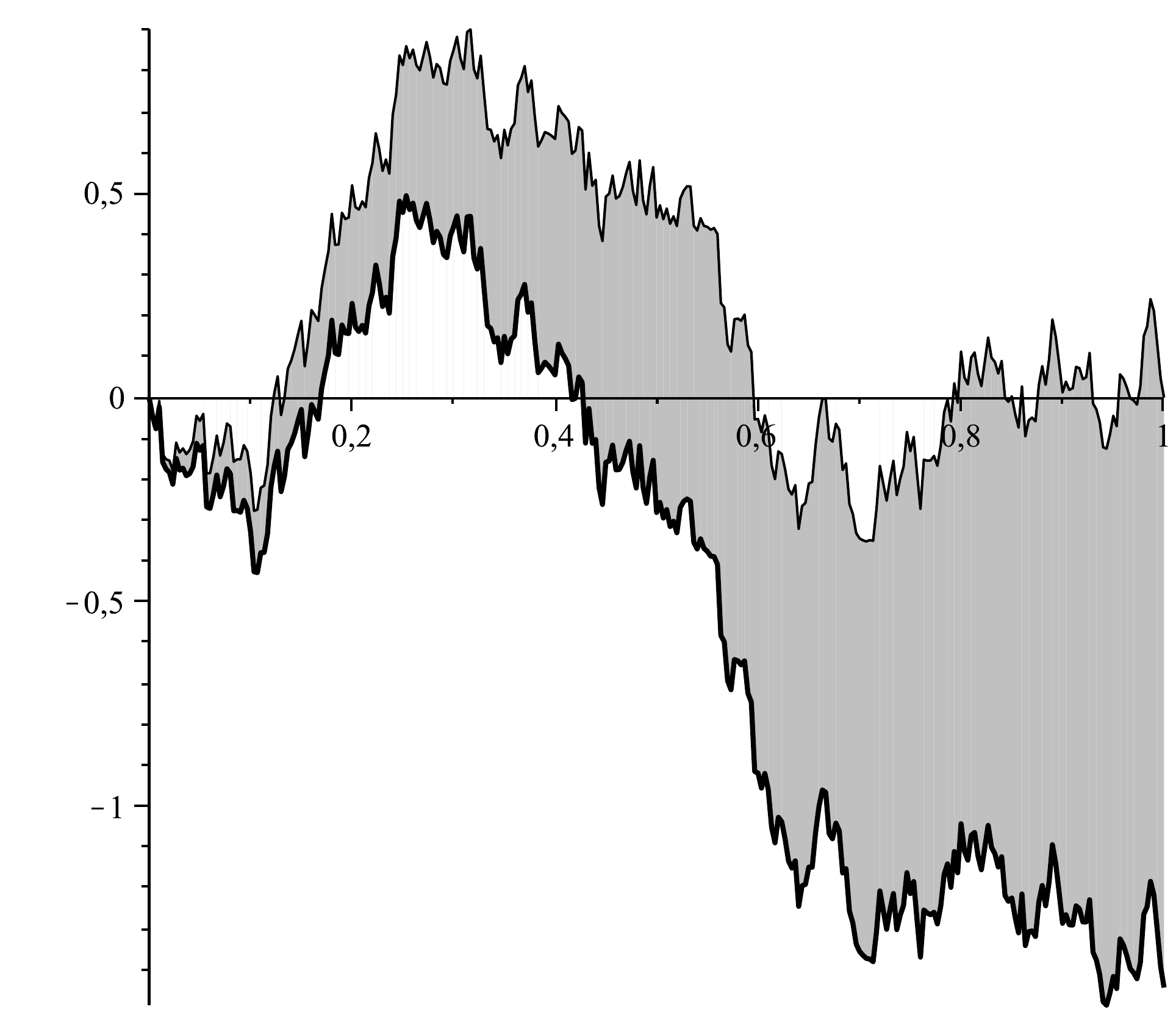}
\includegraphics[scale=0.25]{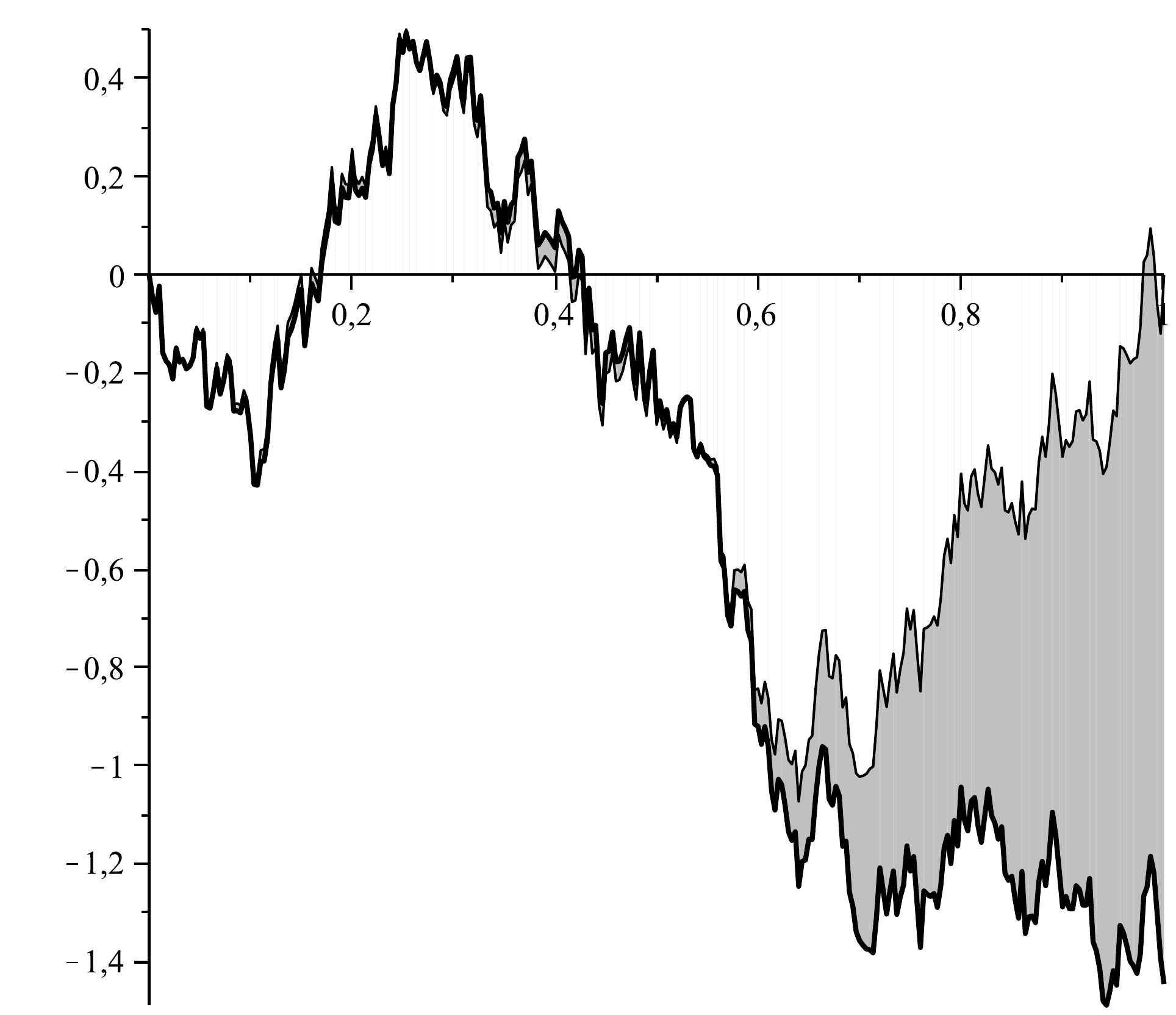}
\includegraphics[scale=0.25]{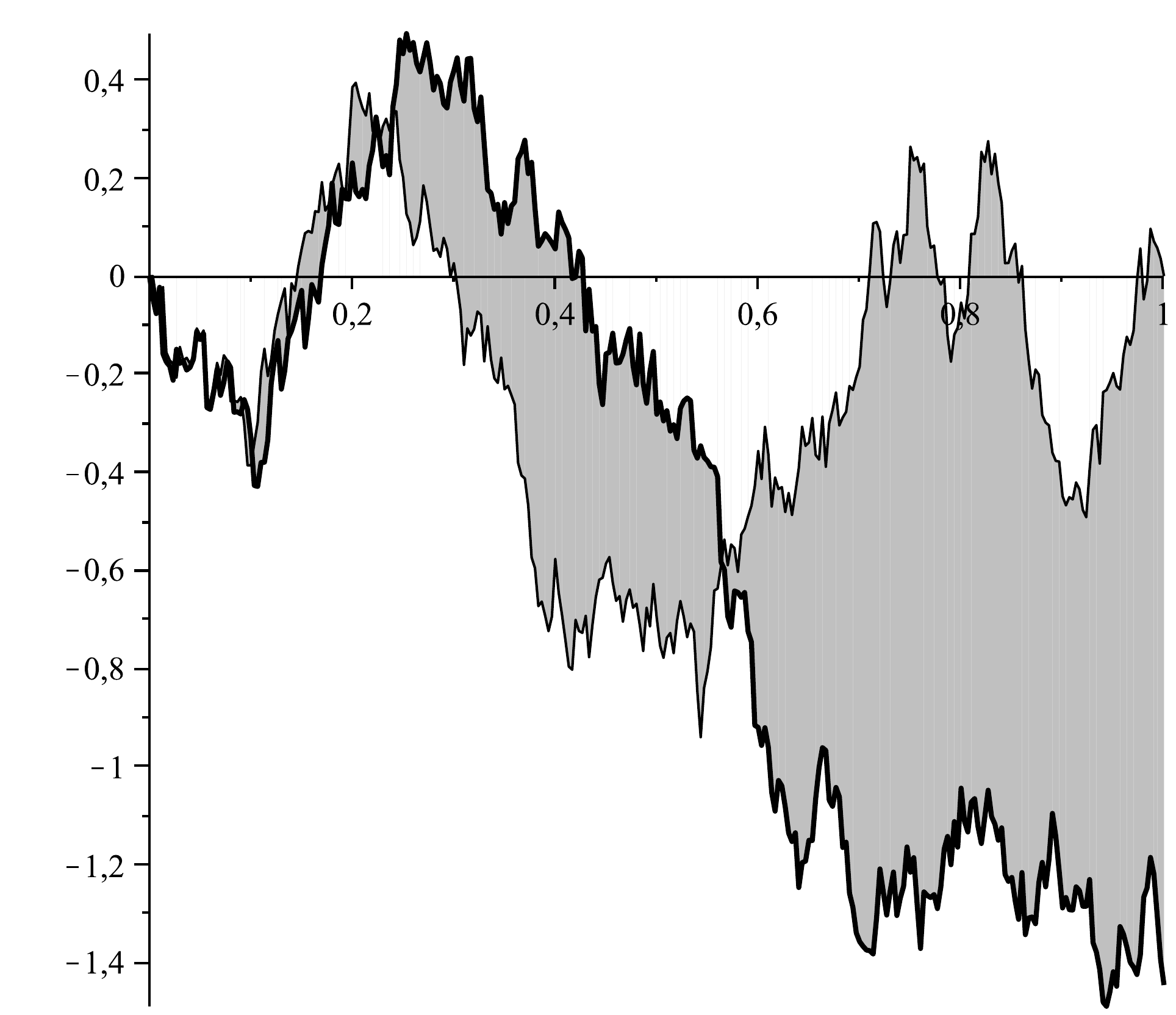}
\includegraphics[scale=0.25]{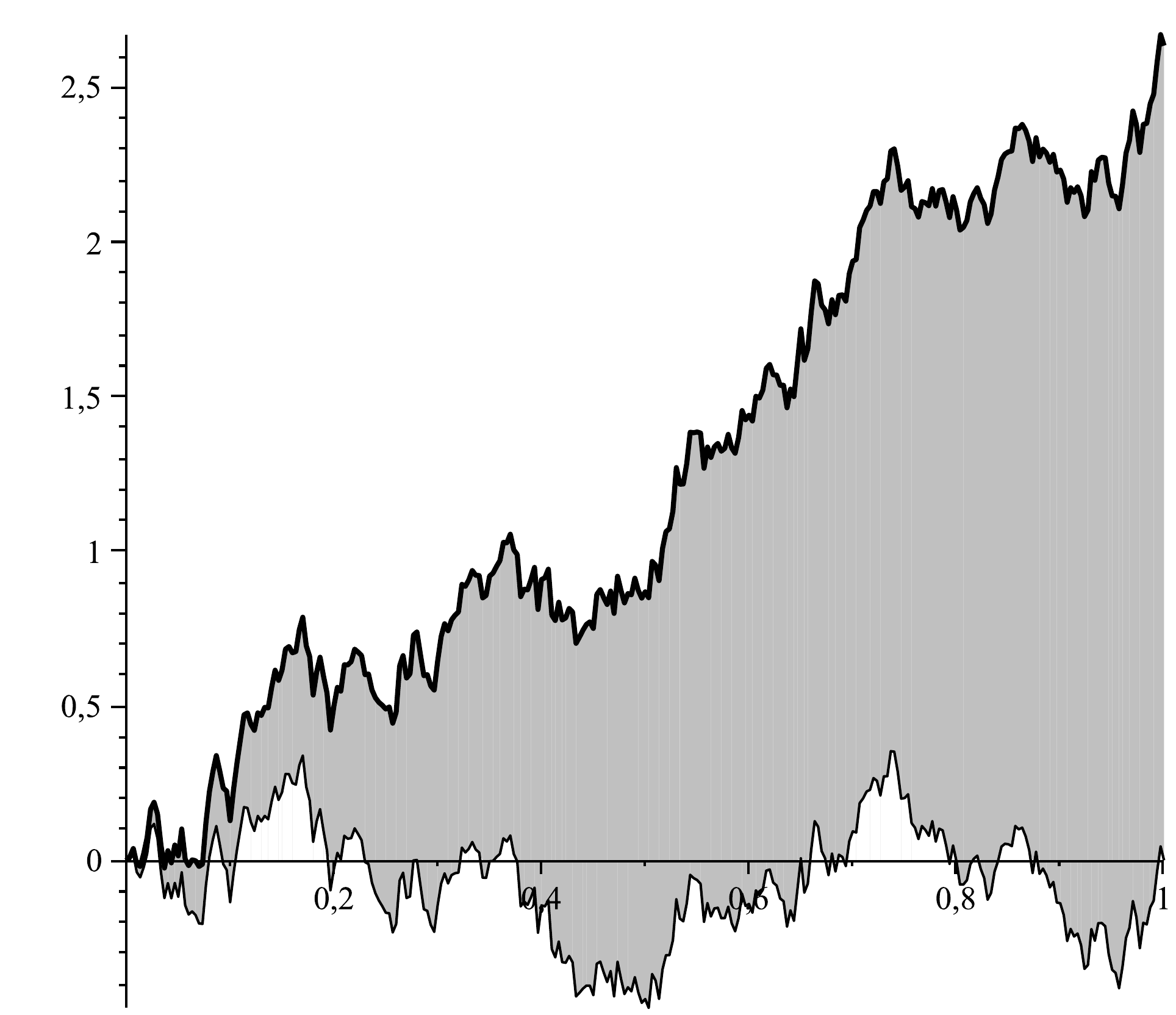}
\includegraphics[scale=0.25]{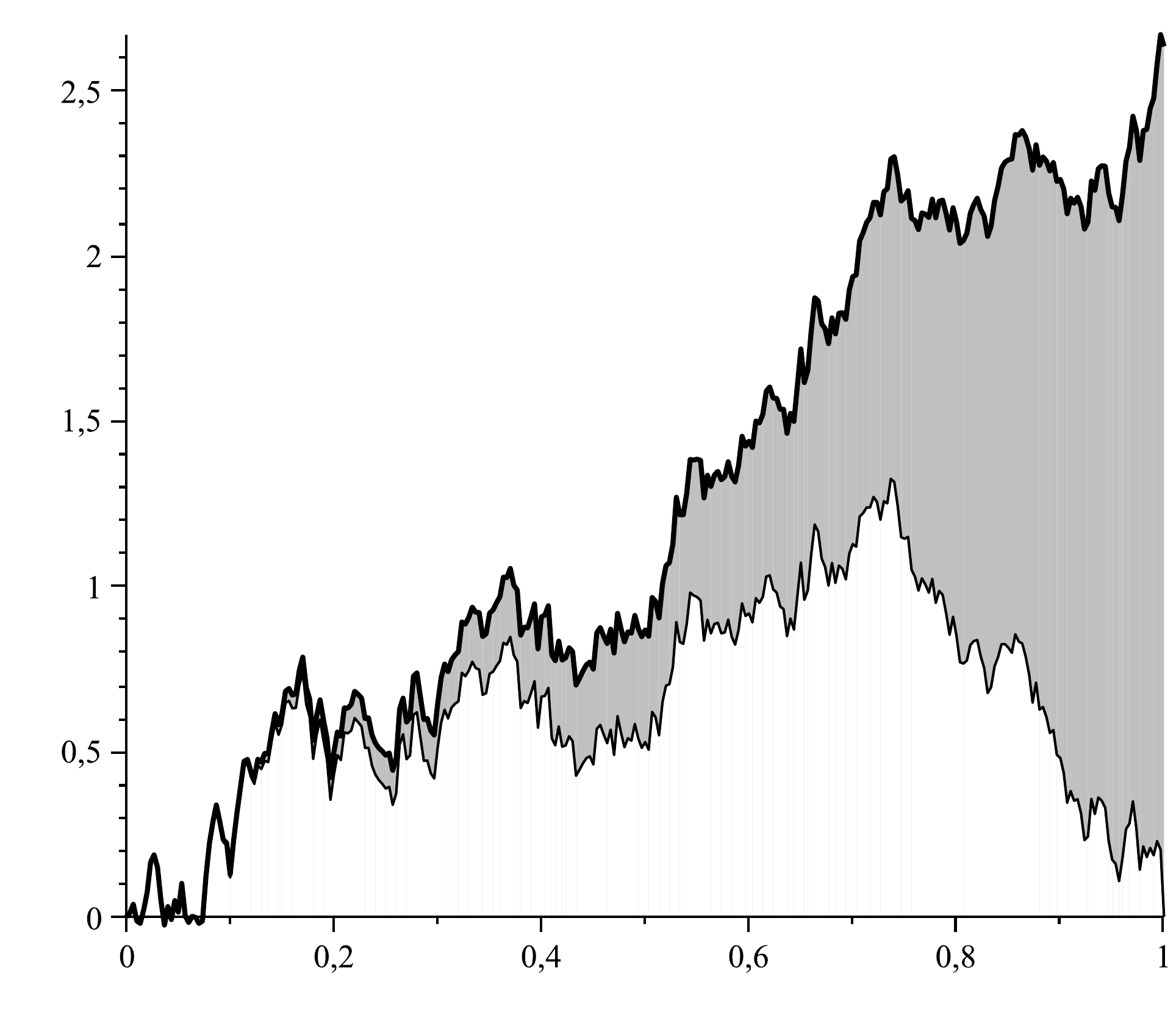}
\includegraphics[scale=0.25]{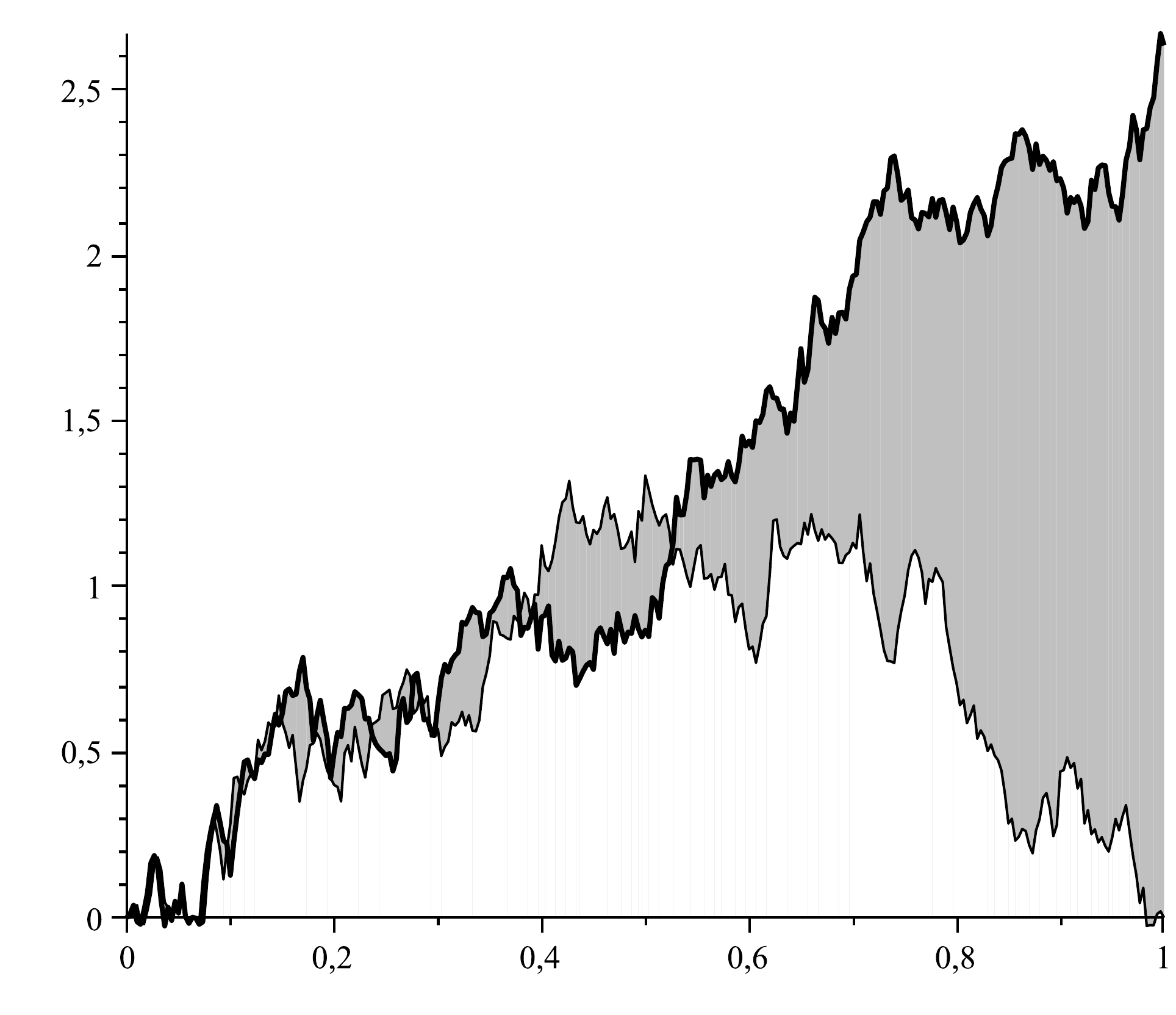}
\caption{\small Two typical sample paths of the Wiener process (rows, thick lines) and its deviations from the
 anticipative version (left column), the integral representation (middle column), and the space-time
 transform (right column) of the Wiener bridge from $0$ to $0$ over the time-interval $[0,1]$.}
\label{Figure1}
\end{figure}
Note that in general the deviation from the space-time transform bridge version is
 hard to compare with the other
 deviations, since $(W_{t}^{\rm st})_{t\in[T/2,T)}$ depends on the non-visible part
 $(W_{t})_{t\in[T,\infty)}$ of the Wiener process.
Our aim is to give quantitative answers to this qualitative behavior observed from simulation studies
 and thus to study the path deviations on $[0,T)$:
\begin{align}
& a+W_{t}-W_{t}^{\rm av}=(a-b)\,\frac{t}{T}+\frac{t}{T}\,W_{T},\nonumber\\
& a+W_{t}-W_{t}^{\rm ir}=(a-b)\,\frac{t}{T}+\int_{0}^t\frac{t-s}{T-s}\,\dd W_{s},\label{devs}\\
& a+W_{t}-W_{t}^{\rm st}=(a-b)\,\frac{t}{T}+\left(W_{t}-\frac{T-t}{T}\,W_{\frac{tT}{T-t}}\right)\nonumber.
\end{align}
Note that the dependence of the path deviations in \eqref{devs} upon the starting and
 endpoint of the bridge ($a$ and $b$) is only via their difference $a-b$.
 Hence without loss of generality we can and will assume $a=0$ in the sequel.

Simulation studies also show that the above typical behavior is reversed in case the endpoint
 $W_{T}$ of the Wiener sample path is close to the prescribed endpoint $b$ of its bridge, namely,
  for such a sample path of the Wiener process the deviation from its anticipative bridge version
  is smaller than from its integral representation of the bridge or from its space-time bridge version;
  see Figure \ref{Figure2}.
\begin{figure}[h]
\includegraphics[scale=0.25]{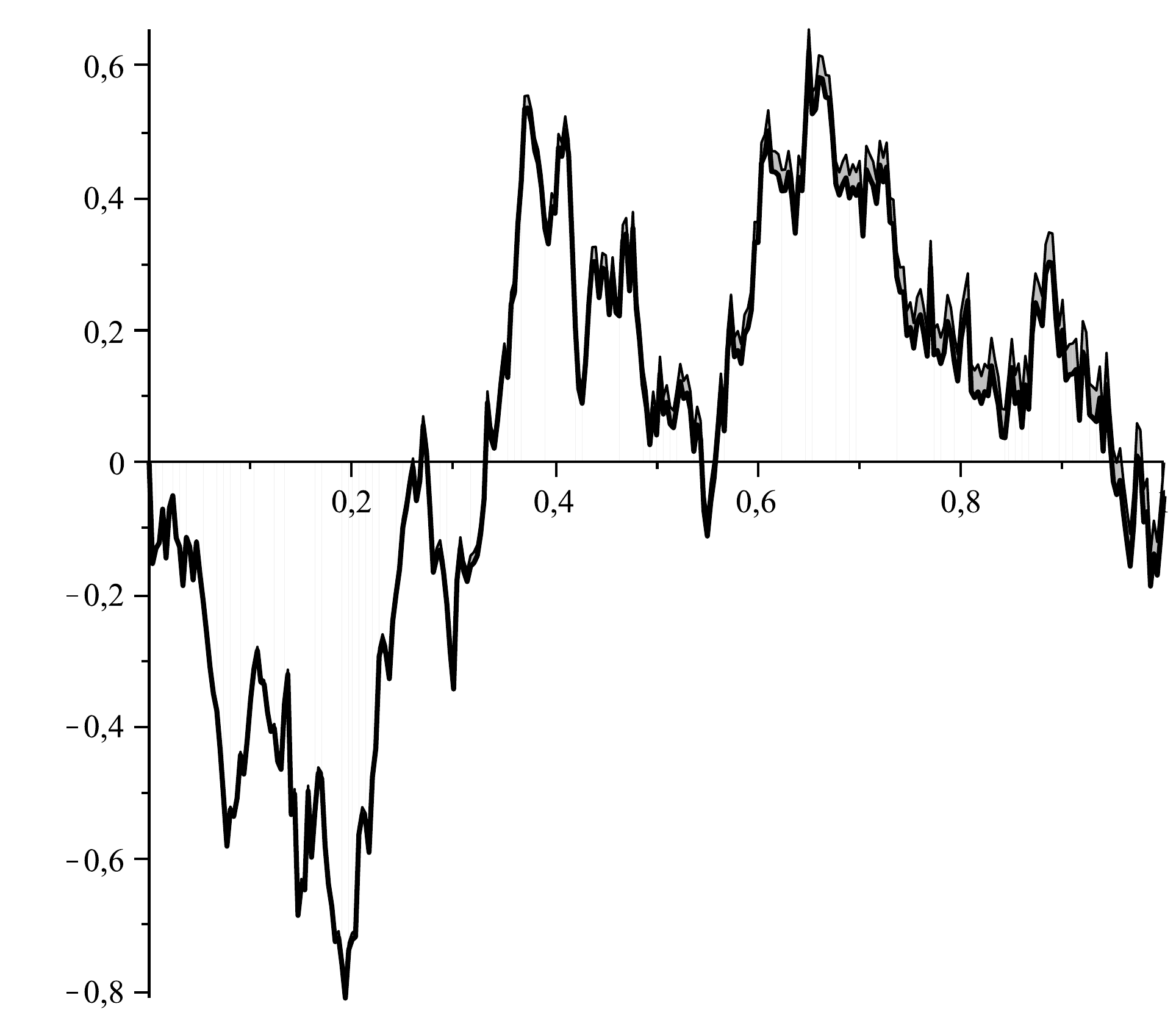}
\includegraphics[scale=0.25]{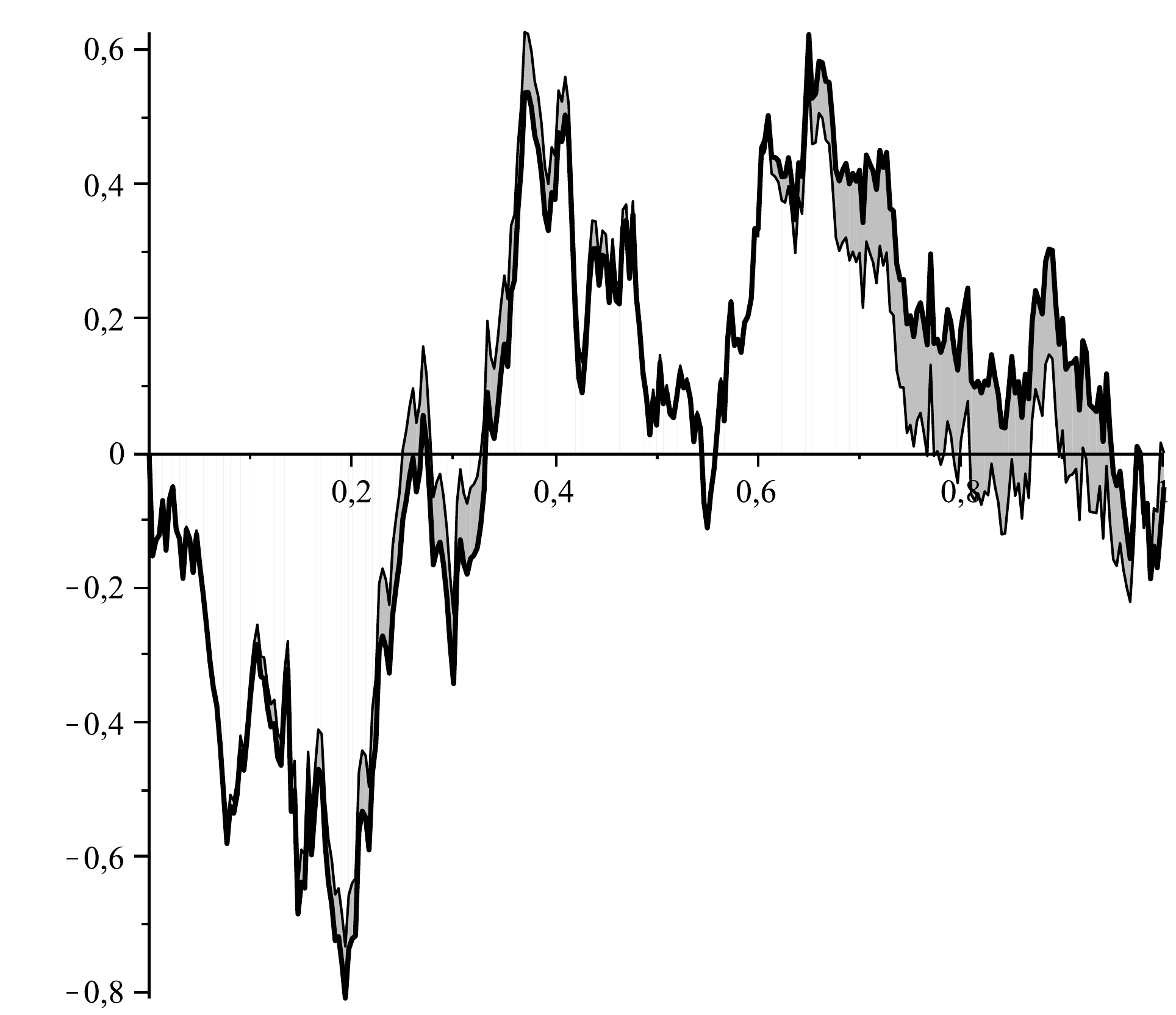}
\includegraphics[scale=0.25]{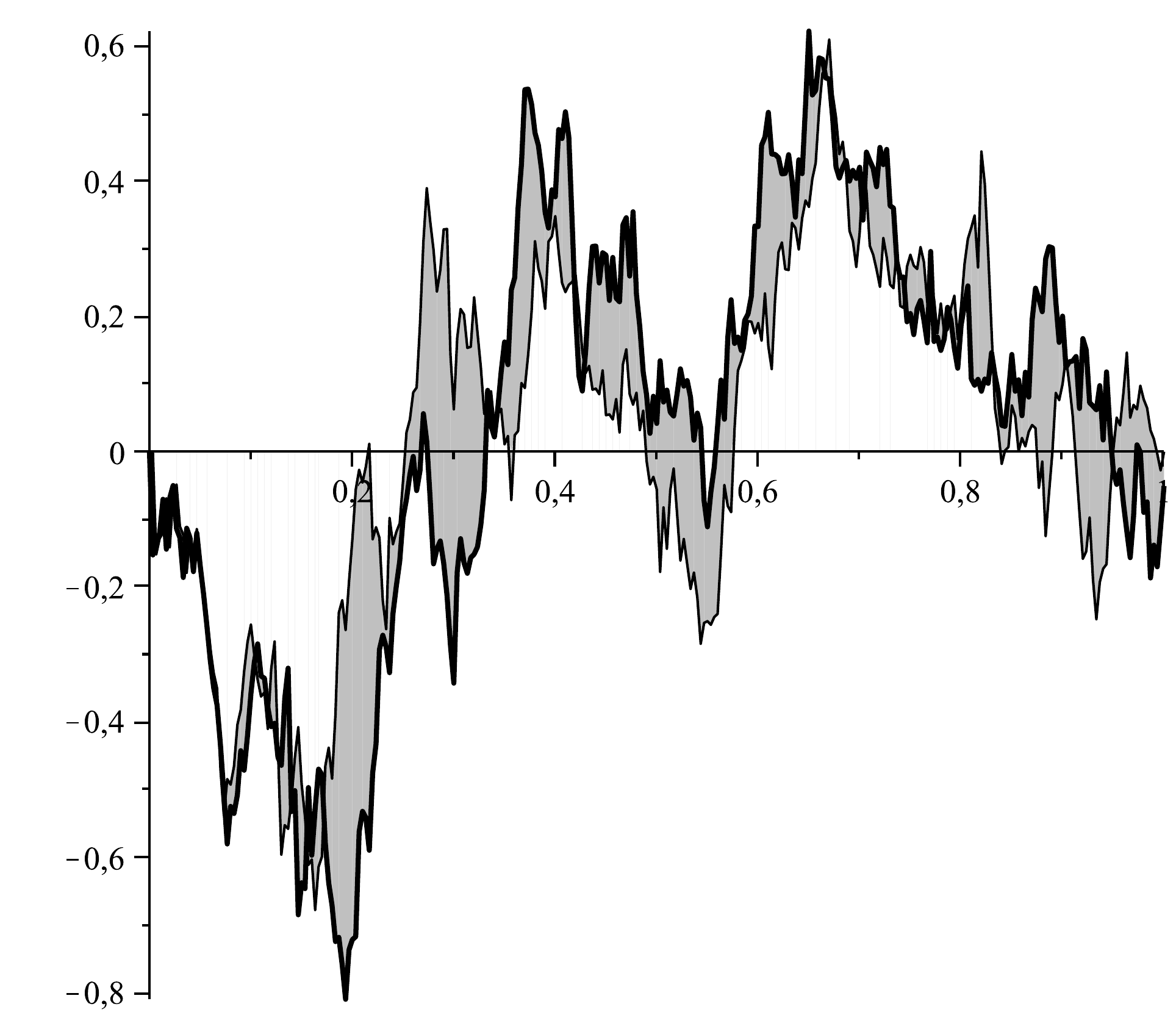}
\caption{\small A sample path of the Wiener process with $W_{1}\approx0$ (thick line) and its deviations  from
 the anticipative version (left), the integral representation (middle), and the space-time transform (right)
 of the Wiener bridge from $0$ to $0$ over the time-interval $[0,1]$.}
\label{Figure2}
\end{figure}
We aim to give quantitative answers to this effect and thus in Section 2
 we will particularly compare the so-called expected $p$-th order sample path deviations
$$
 \Exp\left(\int_{0}^T|W_{t}-W_{t}^{\rm br}|^p\,\dd t\right)
   = \int_{0}^T\Exp\big(|W_{t}-W_{t}^{\rm br}|^p\big)\,\dd t
$$
for $p=1,2$ and in case of \ $p=2$ \ we will explicitly calculate the conditional analogue
$$
  \Exp\left(\int_{0}^T(W_{t}-W_{t}^{\rm br})^2\,\dd t\,\bigg|\,W_{T}=d\right)
   =\int_{0}^T\Exp\big((W_{t}-W_{t}^{\rm br})^2\,\big|\,W_{T}=d\big)\,\dd t
$$
for prescribed endpoints $W_{T}=d$, $d\in\RR$ of the original Wiener process.
In the above formulas, integration over the time-interval $[0,T]$ and taking expectations can be interchanged.
Indeed, since we have continuous sample paths, we can consider monotone approximations of the integrals
 by Riemannian sums with nonnegative summands and then apply the monotone convergence theorem for conditional
 expectations.
 In what follows expected first and second order sample path deviations will be called expected absolute
  and quadratic path deviations, respectively.

We will further show in Section 3 that the above mentioned qualitative behavior of sample path deviations
 is not restricted only to the Wiener process and its bridge versions:
 sample path deviations of the Ornstein-Uhlenbeck process from its bridge versions are also considered.
Here we give some quantitative answers, too, see Theorem \ref{prop10}.

In the Appendix we present an auxiliary result which is used for proving almost sure continuity of the
 integral representation of the Ornstein-Uhlenbeck bridge at the endpoint of the bridge.

Our results are to be seen as paradigmatic examples that give rise for future work concerning more broad questions of how certain pathwise constructions of Gaussian or Markovian bridges can differ, although they obey the same law. The reason for concentrating on the Wiener and on the Ornstein-Uhlenbeck process here is the possibility of giving explicit expressions for
some quantities (such as second moment) related to the path deviations of different bridge versions to the original process through which they are constructed. In particular, the case of an Ornstein-Uhlenbeck process shows that explicit expressions for path deviations can soon become unwieldily. As a future task, one may also address the question of existence of a bridge version that minimizes the distance to the unconditioned stochastic process in a certain sense.
Moreover, one may present other indicators for different sample path behavior of the
 Wiener and Ornstein-Uhlenbeck bridge versions, such as Hellinger distance, and address the question
 for more general process bridges.

To further motivate our study, we point out that similar problems were considered by DasGupta \cite{Das},
 Bharath and Dey \cite{BhaDey} and Balabdaoui and Pitman \cite{BalPit}.
Namely, DasGupta \cite[Theorem 1]{Das} gave an infinite series representation of the expectations
 \[
   \Exp\left(\int_0^\delta\vert W_{t}^{\rm br} - \mu t - W_t \vert\,\dd t \right),
        \qquad \delta\in(0,1],\quad \mu\in\RR,
 \]
 where $(W_t)_{t\in[0,1]}$ and $(W_{t}^{\rm br})_{t\in[0,1]}$ denote respectively a standard
 Wiener process and an independent Wiener bridge with $a=b=0$ and $T=1$. For
 some special values of $\delta$ and $\mu$ the exact values were also calculated.
The motivation of DasGupta for calculating the expectations above is to understand
 whether distinguishing between a Wiener bridge and an independent Wiener process with possible drift
 on the basis of observations at discrete times is intrinsically difficult.
It turned out that distinguishing one from the other is not an easy task.
DasGupta studied the likelihood ratio test for testing the null-hypothesis
 \ $H_0 : X_t = W_t^{\rm br}, \;t\in[0,1]$, \ against the alternative hypothesis
 \ $H_1 : X_t = W_t + \mu t,\, t\in[0,1]$ \ for some \ $\mu\in\RR,$ \ based on discrete
 observations from a process \ $(X_t)_{t\in[0,1]}$.
\ Recently, the question of distinguishing a Wiener process from a Wiener bridge was also considered
 by Bharath and Dey \cite{BhaDey}.
Note that in our setup $(W_t)_{t\in[0,1]}$ and $(W_{t}^{\rm br})_{t\in[0,1]}$ are not independent.
 Hence our results may be useful to answer the question of distinction in case the Wiener bridge
 is constructed by the help of the original Wiener process and not an independent copy.
One can address the same question for Ornstein-Uhlenbeck bridges
 or for more general process bridges.
Our calculations in the Ornstein-Uhlenbeck case can be considered as a first step
 towards the corresponding calculations of Section 2 in DasGupta \cite{Das}.
Balabdaoui and Pitman \cite{BalPit}  gave
 a representation of the maximal difference between a Wiener bridge and its (least) concave
 majorant on the unit interval. As an application, expressions for the distribution,
 density function and moments of this difference were derived.

The presented results might also be applied to the study of animal movements.
Horne et al.~\cite{HGKL} use a two-dimensional Wiener bridge to model the unknown movement
 of an animal between two consecutively observed positions of the animal.
The model is used to investigate questions on the mean
 occupation frequency $\Exp(\frac1T\int_0^T1_A(X_{1,t}^{\rm br},X_{2,t}^{\rm br})\,\dd t)$ in a region
 $A\in\mathcal B(\RR^2)$, where $(X_{1,t}^{\rm br})_{t\in[0,T]}$ and
 $(X_{2,t}^{\rm br})_{t\in[0,T]}$ are independent Wiener bridges
 such that $(X_{1,0}^{\rm br},X_{2,0}^{\rm br})$ and
 $(X_{1,T}^{\rm br},X_{2,T}^{\rm br})$ are the starting and ending positions of the animal
 at time $0$ and $T$, respectively.
If the region $A$ depends on the original (independent) Wiener processes
 $(X_{1,t})_{t\in[0,T]}$, $(X_{2,t})_{t\in[0,T]}$, e.g., for questions concerning the closeness
 of the animal's path to the path of a Wiener process, our results show that the expected occupation
 frequency heavily depends on the chosen version of the bridge.

\section{Path deviation of the Wiener process from its bridges}

\subsection{An indicator for different sample path behavior of Wiener bridge versions}

A first indicator for different sample path behavior of the bridge versions
 is the correlation function $\varrho(W_{t}^{\rm br},W_{t})$ of these bridge versions and
 the original Wiener process.
Note that \ $(W_{t}^{\rm br},W_{t})_{t\in[0,T]}$ \ is a two-dimensional Gauss process and the
 correlation coefficient of the two coordinates is given in the next proposition.

\begin{prop}\label{correlation}
For all $t\in(0,T)$, we have
$$
 \varrho(W_{t}^{\rm av},W_{t})=\varrho(W_{t}^{\rm st},W_{t})=\sqrt{\frac{T-t}{T}}
   \quad\text{ and }\quad
 \varrho(W_{t}^{\rm ir},W_{t})=\frac{\sqrt{T(T-t)}}{t}\,\log\frac{T}{T-t}.
 $$
\end{prop}
\begin{proof}
By \eqref{bridgecov}, we get for every $0\leq t\leq T$
$$\Var(W_{t}^{\rm br})=\Cov(W_{t}^{\rm br},W_{t}^{\rm br})=t\frac{T-t}{T}.$$
We easily calculate for every $0\leq t<T$
\begin{align*}
\Cov(W_{t}^{\rm av},W_{t}) & =\Cov(W_{t},W_{t})-\frac{t}{T}\,\Cov(W_{T},W_{t})=t-\frac{t^2}{T}=t\frac{T-t}{T},\\
\Cov(W_{t}^{\rm ir},W_{t}) & =\Cov\left(\int_{0}^t\frac{T-t}{T-s}\,\dd W_{s},
           \int_{0}^t 1\,\dd W_{s}\right)\!=\!\int_{0}^t\frac{T-t}{T-s}\,\dd s=(T-t)\log\frac{T}{T-t},
\intertext{and}
\Cov(W_{t}^{\rm st},W_{t}) & =\Cov\left(\frac{T-t}{T}W_{\frac{tT}{T-t}},W_{t}\right)=t\frac{T-t}{T}.
\end{align*}
Thus we get for every \ $0<t<T$,
\begin{align}\label{help1}
 \varrho(W_{t}^{\rm av},W_{t})=\frac{t\frac{T-t}{T}}{\sqrt{\frac{t}{T}\,(T-t)\cdot t}}
   =\sqrt{\frac{T-t}{T}}=\varrho(W_{t}^{\rm st},W_{t})
\end{align}
and
$$\varrho(W_{t}^{\rm ir},W_{t})
   =\frac{(T-t)\log\frac{T}{T-t}}{\sqrt{\frac{t}{T}\,(T-t)\cdot t}}
   =\frac{\sqrt{T(T-t)}}{t}\,\log\frac{T}{T-t}
$$
concluding the proof.
\end{proof}

\begin{remark}
For all \ $T\in(0,\infty)$, \ the function \ $(0,T)\ni t\mapsto \varrho(W_{t}^{\rm br},W_{t})$ \ is
 strictly decreasing. For the anticipative version and space-time transform,
 it is an immediate consequence of \eqref{help1}.
For the integral representation, it is enough to check that
 \begin{align*}
   \frac{\partial}{\partial t}\!\left(\!\frac{\sqrt{T(T-t)}}{t}\,\log\frac{T}{T-t}\right)
     =\frac{-\frac{t\sqrt{T}}{2\sqrt{T-t}}\log\left(\frac{T}{T-t}\right)
            +\frac{t\sqrt{T}}{\sqrt{T-t}}
            -\sqrt{T(T-t)}\log\left(\frac{T}{T-t}\right)}
     {t^2}
     < 0
 \end{align*}
 for all $t\in(0,T)$, which is equivalent to show that
 \[
   h(t):=\left(T-\frac{t}{2}\right)\log\left(\frac{T-t}{T}\right) + t < 0, \qquad t\in(0,T).
 \]
Using that \ $\log(1-x) = -\sum_{k=1}^\infty \frac{x^k}{k}$ \ for all \ $-1<x<1$, \ we get
 \begin{align*}
   h(t) & = -\left(T-\frac{t}{2}\right)\sum_{k=1}^\infty \frac{1}{k}\left(\frac{t}{T}\right)^k
            + t
          = t - \sum_{k=1}^\infty\frac{t^k}{kT^{k-1}}
            + \sum_{k=1}^\infty\frac{t^{k+1}}{2kT^k}  \\
        & = - \sum_{k=1}^\infty\left(\frac{t^{k+1}}{(k+1)T^k} - \frac{t^{k+1}}{2kT^k} \right)
          = - \sum_{k=1}^\infty\left(\frac{1}{k+1} - \frac{1}{2k}\right)\frac{t^{k+1}}{T^k} \\
         & = - \sum_{k=1}^\infty \frac{k-1}{2k(k+1)}\frac{t^{k+1}}{T^k}
           <0,\qquad t\in(0,T).
 \end{align*}

\noindent Note also that
 $\varrho(W_{t}^{\rm br},W_{t})\to1$ as $t\downarrow0$, and $\varrho(W_{t}^{\rm br},W_{t})\to0$
 as $t\uparrow T$.
Hence $W_{t}^{\rm br}$ and $W_{t}$, $t\in(0,T)$, are positively correlated for all bridge versions.
Moreover,
\begin{equation}\label{compare}
 \frac{\sqrt{T(T-t)}}{t}\,\log\frac{T}{T-t}>\sqrt{\frac{T-t}{T}},\qquad t\in(0,T).
\end{equation}
Indeed, \eqref{compare} is equivalent to $-\frac{t}{T}>\log\left(1-\frac{t}{T}\right)$ \ for all \ $t\in(0,T)$,
 which follows by \ $\log(1-x)\leq -x$ \ for all \ $0\leq x<1$.

\noindent Hence the integral representation is more positively correlated to the original process
 than the anticipative version and the space-time transform.
\proofend
\end{remark}

\subsection{Gauss and conditional Gauss distribution of path deviations}
           \label{subsection_gauus_cond_gauss}

First we study the distribution of the path deviation \ $W_t-W_{t}^{\rm br}$, $t\in[0,T)$.

\begin{prop}\label{prop1}
Let \ $(W_{t}^{\rm br})_{t\in[0,T]}$ \ be a Wiener bridge from \ $0$ \ to \ $b$ \
 over the time-interval \ $[0,T]$, \ where \ $b\in\RR$.
Then for all \ $t\in[0,T)$, \ the path deviation $W_{t}-W_{t}^{\rm br}$ is
 normally distributed with mean \ $\Exp(W_{t}-W_{t}^{\rm br})=-b\frac{t}{T}$ \ and with variance
 \begin{align*}
   \Var(W_{t}-W_{t}^{\rm av})& = \Var(W_{t}-W_{t}^{\rm st}) = \frac{t^2}{T},\\
   \Var(W_{t}-W_{t}^{\rm ir})& = t\left(1+\frac{T-t}{T}\right)+2(T-t)\log\frac{T-t}{T}=:\sigma^2(t).
 \end{align*}
\end{prop}

\begin{proof}
With $a=0$, by \eqref{devs}, for every $0\leq t<T$ the path deviation $W_{t}-W_{t}^{\rm br}$ is
 normally distributed with  mean $\Exp(W_{t}-W_{t}^{\rm br})=-b\frac{t}{T}$ and with variance
\begin{align*}
 \Var(W_{t}-W_{t}^{\rm av})& =\Var\left(\frac{t}{T}W_T\right)=\frac{t^2}{T},\\
 \Var(W_{t}-W_{t}^{\rm ir})&
  =\Var\left(\int_0^t\frac{t-s}{T-s}\,\dd W_s\right)=\int_0^t\left(\frac{t-s}{T-s}\right)^2\,\dd s\\
 & =\int_0^t\left(1-\frac{T-t}{T-s}\right)^2\,\dd s
  =\int_0^t1-2\frac{T-t}{T-s}+\left(\frac{T-t}{T-s}\right)^2\,\dd s\\
 & =t+2(T-t)\log\frac{T-t}{T}+(T-t)^2\left(\frac1{T-t}-\frac1T\right)
   =\sigma^2(t),
  \intertext{and} \Var(W_{t}-W_{t}^{\rm st})
 & =\Var\left(W_t-\frac{T-t}{T}W_{\frac{tT}{T-t}}\right)
   =\Var\left(-\frac{T-t}{T}(W_{\frac{tT}{T-t}}-W_t)+\frac{t}{T}W_t\right)\\
 & =\left(\frac{T-t}{T}\right)^2\left(\frac{tT}{T-t}-t\right)+\frac{t^3}{T^2}
 =\frac{(T-t)t^2+t^3}{T^2}=\frac{t^2}{T}
\end{align*}
concluding the proof.
\end{proof}

By Proposition \ref{prop1}, for every $0<t<T$, the variance of the path deviation of the
 integral representation from the original Wiener process is smaller
 than those of the anticipative version or the space-time transform, since we have
 $\sigma^2(t)=2t-\frac{t^2}{T}+2(T-t)\log(1-\frac{t}{T})$ and thus
 \begin{equation}\label{compare_2}
  \sigma^2(t)<\frac{t^2}{T}, \qquad t\in(0,T).
 \end{equation}
Indeed, \eqref{compare_2} is equivalent to $-\frac{t}{T}>\log\left(1-\frac{t}{T}\right)$ \ for all \  $t\in(0,T)$, \
 which holds, since \ $\log(1-x)\leq -x$ \ for all \ $0\leq x<1$.

Next we examine the conditional distribution of the path deviation \ $W_t-W_{t}^{\rm br}$ \
 given the endpoint \ $W_T$.

\begin{prop}\label{prop2}
Let \ $(W_{t}^{\rm br})_{t\in[0,T]}$ \ be a Wiener bridge from \ $0$ \ to \ $b$ \
 over the time-interval \ $[0,T]$, \ where \ $b\in\RR$.
Then for all \ $t\in[0,T)$ \ and \ $d\in\RR$, \
 the conditional distribution of the path deviation
 \ $W_t - W_{t}^{\rm br}$ \ given \ $W_T=d$ \ is normal with mean
 \begin{align}\label{help_cond_gauss1}
   & \Exp(W_{t}-W_{t}^{\rm av}\mid W_{T}=d) = (d-b)\frac{t}{T},\\ \label{help_cond_gauss2}
   & \Exp(W_{t}-W_{t}^{\rm ir}\mid W_{T}=d)
       = (d-b)\frac{t}{T}+d\frac{T-t}{T}\log\frac{T-t}{T},\\ \label{help_cond_gauss3}
   & \Exp(W_{t}-W_{t}^{\rm st}\mid W_{T}=d) =-b\frac{t}{T}+\frac{d}{T}(2t-T)\cdot 1_{[\frac{T}{2},T)}(t),
 \end{align}
 and with variance
 \begin{align} \label{help_cond_gauss4}
   & \Var(W_{t}-W_{t}^{\rm av}\mid W_{T}=d) = 0,\\\label{help_cond_gauss5}
   & \Var(W_{t}-W_{t}^{\rm ir}\mid W_{T}=d)
      = 2t\frac{T-t}{T}+2\frac{(T-t)^2}{T}\log\frac{T-t}{T}\\\nonumber
   &\phantom{\Var(W_{t}-W_{t}^{\rm ir}\mid W_{T}=d)=\;}
         - \frac{(T-t)^2}{T}\left(\log\frac{T-t}{T}\right)^2, \\\label{help_cond_gauss6}
  & \Var(W_{t}-W_{t}^{\rm st}\mid W_{T}=d) =\frac{t^2}{T}-\frac{(2t-T)^2}{T}\cdot 1_{[\frac{T}{2},T)}(t).
 \end{align}
\end{prop}

\begin{proof}
For all \ $0\leq t<T$, the joint distribution $(W_{t}-W_{t}^{\rm br},W_{T})$ of the path
 deviation and the endpoint is a two-dimensional normal distribution and, by Theorem 2 and
 Problem 5 in Chapter II, \S13 of Shiryaev \cite{Shi}, it is known that the conditional
 distribution of $W_{t}-W_{t}^{\rm br}$ given $W_{T}=d$ \ is normal with mean
\begin{equation}\label{condexp}
\Exp\left(W_{t}-W_{t}^{\rm br}\right)
   +\frac{d-\Exp\left(W_{T}\right)}{\Var\left(W_{T}\right)}\,\Cov\left(W_{t}-W_{t}^{\rm br},W_{T}\right)
\end{equation}
and with variance
\begin{equation}\label{condvar}
\Var\left(W_{t}-W_{t}^{\rm br}\right)
    -\frac{\left(\Cov\left(W_{t}-W_{t}^{\rm br},W_{T}\right)\right)^2}{\Var\left(W_{T}\right)}\,.
\end{equation}
Here we have
 $$
  \Cov(W_{t}-W_{t}^{\rm av},W_{T})=\Cov\left(\frac{t}{T}W_{T},W_{T}\right)=t,\qquad t\in[0,T),
 $$
 and thus \eqref{condexp}, \eqref{condvar} and Proposition \ref{prop1} yield that
\begin{align*}
 &\Exp(W_{t}-W_{t}^{\rm av}\mid W_{T}=d)  =-b\frac{t}{T}+\frac{d}{T}t=(d-b)\frac{t}{T},\\
 &\Var(W_{t}-W_{t}^{\rm av}\mid W_{T}=d)  =\frac{t^2}{T}-\frac{t^2}{T}=0.
\end{align*}
We note that the above formulae follow immediately, since in case of \ $W_T=d$, \ we have
 \ $W_{t}-W_{t}^{\rm av}=(d-b)\frac{t}{T}$.
Further, we have
\begin{align*}
\Cov(W_{t}-W_{t}^{\rm ir},W_{T})
 &  =\Cov\left(\int_{0}^t\frac{t-s}{T-s}\,\dd W_{s},\int_{0}^T 1\,\dd W_{s}\right)
   =\int_{0}^t\frac{t-s}{T-s}\,\dd s\\
 & =\int_{0}^t\left(1-\frac{T-t}{T-s}\right)\,\dd s=t+(T-t)\log\frac{T-t}{T},
\end{align*}
 and thus \eqref{condexp}, \eqref{condvar} and Proposition \ref{prop1} yield that
\begin{align*}
 \Exp(W_{t}-W_{t}^{\rm ir}\mid W_{T}=d) & = -b\frac{t}{T}+\frac{d}{T}\left(t+(T-t)\log\frac{T-t}{T}\right),\\
 \Var(W_{t}-W_{t}^{\rm ir}\mid W_{T}=d) & = \sigma^2(t)-\frac{\left(t+(T-t)\log\frac{T-t}{T}\right)^2}T\\
 & =2t-\frac{t^2}{T}+2(T-t)\log\frac{T-t}{T}\\
 &\phantom{=}-\frac{t^2}{T}-2(T-t)\frac{t}{T}\log\frac{T-t}{T}
            -\frac{(T-t)^2}{T}\left(\log\frac{T-t}{T}\right)^2.
\end{align*}
This implies \eqref{help_cond_gauss2} and \eqref{help_cond_gauss5}.
Finally, we have
\begin{align*}
\Cov(W_{t}-W_{t}^{\rm st},W_{T}) & =\Cov(W_{t},W_{T})-\frac{T-t}{T}\Cov\left(W_{\frac{tT}{T-t}},W_{T}\right)\\
& =t-\frac{T-t}{T}\min\left(\frac{tT}{T-t},T\right)\\
& =\begin{cases}
     t-\frac{T-t}{T}\frac{tT}{T-t}=0
          & \text{ if }0\leq t\leq\frac{T}{2},\\
     t-\frac{T-t}{T}T=2t-T & \text{ if }\frac{T}{2}\leq t<T,
    \end{cases}
\end{align*}
 and thus \eqref{condexp}, \eqref{condvar} and Proposition \ref{prop1} yield
 \eqref{help_cond_gauss3} and \eqref{help_cond_gauss6}.
\end{proof}

\subsection{Comparison of tail functions}

The tail of a normally distributed random variable $Y_{\mu,\sigma^2}$ with mean $\mu\in\RR$ and
 with variance $\sigma^2>0$ has the form
\begin{align*}
 T_{\mu,\sigma^2}(x)& = \PP\{|Y_{\mu,\sigma^2}|>x\}=1-\PP \{-x\leq Y_{\mu,\sigma^2}\leq x\}\\
  & =1-\PP\left\{\frac{-x-\mu}{\sigma}\leq \frac{Y_{\mu,\sigma^2}-\mu}{\sigma}
   \leq\frac{x-\mu}{\sigma}\right\}=1-\Phi\left(\frac{x-\mu}{\sigma}\right)
            +\Phi\left(\!-\frac{x+\mu}{\sigma}\right)\\
  & =1-\Phi\left(\frac{\mu+x}{\sigma}\right)+\Phi\left(\frac{\mu-x}{\sigma}\right),
     \qquad x>0,
\end{align*}
where $\Phi$ denotes the standard normal distribution function.
Since, by Proposition \ref{prop1}, \ $\Exp(W_t-W_{t}^{\rm br})=-b\frac{t}{T}$, $t\in[0,T)$,
 \ if we want to use the monotonicity in \eqref{compare_2} to show different behavior of
 the tails of the deviations $W_{t}-W_{t}^{\rm br}$, then this tail function should be
 an increasing function in $\sigma>0$ for every fixed $x>0$ and fixed
 $\mu:=-b\frac{t}{T}\in\mathbb R$. We have
\begin{align*}
\frac{\partial}{\partial\sigma}T_{\mu,\sigma^2}(x)& =\frac{x+\mu}{\sigma^2}\Phi'\left(\frac{\mu+x}{\sigma}\right)+\frac{x-\mu}{\sigma^2}\Phi'\left(\frac{\mu-x}{\sigma}\right)\\
& =\frac{x+\mu}{\sigma^2\sqrt{2\pi}}\exp\left(-\frac12\left(\frac{x+\mu}{\sigma}\right)^2\right)+\frac{x-\mu}{\sigma^2\sqrt{2\pi}}\exp\left(-\frac12\left(\frac{x-\mu}{\sigma}\right)^2\right)\\
& =\sqrt{\frac2\pi}\,\frac1{\sigma}\exp\left(-\frac12\frac{x^2+\mu^2}{\sigma^2}\right)\,\left(\frac{x}{\sigma}\cosh\left(\frac{x\mu}{\sigma^2}\right)-\frac{\mu}{\sigma}\sinh\left(\frac{x\mu}{\sigma^2}\right)\right).
\end{align*}
In case $0<x<|\mu|$, $\mu\ne 0$, we have
$$
  \frac{x}{\sigma}\cosh\left(\frac{x\mu}{\sigma^2}\right)
     -\frac{\mu}{\sigma}\sinh\left(\frac{x\mu}{\sigma^2}\right)
  =\frac{x-\mu}{2\sigma}\exp\left(\frac{x\mu}{\sigma^2}\right)
   +\frac{x+\mu}{2\sigma}\exp\left(-\frac{x\mu}{\sigma^2}\right)
  \to-\infty,
$$
as $\sigma\downarrow0$.
This shows that in general the tail function $T_{\mu,\sigma^2}(x)$ is not increasing in $\sigma>0$ and
 thus is in general not helpful to analyze the different behavior of path deviations.

In special situations such as $b=0=\mu$ it is evident that
 $(0,\infty)\ni\sigma\mapsto T_{0,\sigma^2}(x)$ is strictly
 increasing for every $x>0$. In this special case it follows immediately from the formula
 $\Exp(|Y_{0,\sigma^2}|^p)=\int_{0}^\infty x^{p-1}T_{0,\sigma^2}(x)\,\dd x$, $p\geq 1$, and \eqref{compare_2}
 that for every $p\geq1$ and $0<t<T$ we have
$$
  \Exp\left(|W_{t}-W_{t}^{\rm ir}|^p\right)<\Exp\left(|W_{t}-W_{t}^{\rm av}|^p\right)
    =\Exp\left(|W_{t}-W_{t}^{\rm st}|^p\right).
$$
As a further consequence we get for every $p\geq1$
$$\Exp\left(\int_{0}^T|W_{t}-W_{t}^{\rm ir}|^p\,\dd t\right)
  <\Exp\left(\int_{0}^T|W_{t}-W_{t}^{\rm av}|^p\,\dd t\right)
     =\Exp\left(\int_{0}^T|W_{t}-W_{t}^{\rm st}|^p\,\dd t\right).
$$
We will now show for $p=1$ and $p=2$ that these relations are also true in the general case with $b\not=0$,
 see Subsection \ref{subsection_exp_abs_path_dev}.
In addition, we will get explicit expressions for the expected (conditional) path
 deviations in the case $p=2$.
The reason for not considering a general \ $p\in\NN$ \ is that we just want to demonstrate
 the phenomenon that the bridge versions have different sample path behavior.
We also note that the calculations for a general \ $p\in\NN$ \ would be more complicated.

\subsection{Expected absolute, quadratic and conditional quadratic
            path deviations}\label{subsection_exp_abs_path_dev}

First we study the $L^1$-norm \ $\Exp(|W_{t}-W_{t}^{\rm br}|)$ \ of the path deviations
 \ $W_t - W_{t}^{\rm br}$.

\begin{lemma}\label{prop3}
Let \ $(W_{t}^{\rm br})_{t\in[0,T]}$ \ be a Wiener bridge from \ $0$ \ to \ $b$ \
 over the time-interval \ $[0,T]$, \ where \ $b\in\RR$.
Then for all \ $t\in(0,T)$,
 \begin{align*}
    \Exp\left(|W_t-W_{t}^{\rm av}|\right) = \Exp\left(|W_t-W_{t}^{\rm st}|\right)
       > \Exp\left(|W_t-W_{t}^{\rm ir}|\right).
 \end{align*}
\end{lemma}

\begin{proof}
For a normally distributed random variable $Y_{\mu,\sigma^2}$ with mean $\mu$ and
 with variance $\sigma^2>0$ we have
\begin{align*}
\Exp(|Y_{\mu,\sigma^2}|)
 & =\int_{-\infty}^\infty |x|\,\frac{1}{\sqrt{2\pi\sigma^2}}
     \exp\left(-\frac12\frac{(x-\mu)^2}{\sigma^2}\right)\,\dd x\\
& =\left(\int_0^\infty-\int_{-\infty}^0\right) x\,\frac{1}{\sqrt{2\pi\sigma^2}}
      \exp\left(-\frac12\frac{(x-\mu)^2}{\sigma^2}\right)\,\dd x\\
& =\left(\int_0^\infty-\int_{-\infty}^0\right)
   \frac{x-\mu}{\sigma}\,\frac{1}{\sqrt{2\pi}}\exp\left(-\frac12\frac{(x-\mu)^2}{\sigma^2}\right)\,\dd x\\
& \phantom{=}+\mu\left(\int_0^\infty-\int_{-\infty}^0\right)
 \frac{1}{\sqrt{2\pi\sigma^2}}\exp\left(-\frac12\frac{(x-\mu)^2}{\sigma^2}\right)\,\dd x.
\end{align*}
By change of variables $y=\frac12(\frac{x-\mu}{\sigma})^2$ and $z=\frac{x-\mu}{\sigma}$, we get
\begin{align}\nonumber
 \Exp(|Y_{\mu,\sigma^2}|) & =2\sigma\int_{\frac12(\frac{\mu}{\sigma})^2}^\infty
  \frac{1}{\sqrt{2\pi}}\exp\left(-y\right)\,\dd y
    +\mu\left(\int_{-\frac{\mu}{\sigma}}^\infty-\int_{-\infty}^{-\frac{\mu}{\sigma}}\right)
    \frac{1}{\sqrt{2\pi}}\exp\left(-\frac12z^2\right)\,\dd z\\ \label{help_exp_abs_path_dev2}
 & = 2\sigma\,\frac{1}{\sqrt{2\pi}}\exp\left(-\frac{\mu^2}{2\sigma^2}\right)
      +\mu\left(1-2\Phi\left(-\frac{\mu}{\sigma}\right)\right)\\\nonumber
 & = 2\sigma\,\Phi'\left(\frac{\mu}{\sigma}\right)+\mu\left(2\Phi\left(\frac{\mu}{\sigma}\right)-1\right).
\end{align}
Differentiation with respect to $\sigma>0$, using $\Phi''(x)=-x\Phi'(x)$ yields
\begin{align*}
 \frac{\partial}{\partial\sigma}\Exp(|Y_{\mu,\sigma^2}|) & =2\Phi'\left(\frac{\mu}{\sigma}\right)
   -2\frac{\mu}{\sigma}\Phi''\left(\frac{\mu}{\sigma}\right)
   -2\frac{\mu^2}{\sigma^2}\Phi'\left(\frac{\mu}{\sigma}\right)=2\Phi'\left(\frac{\mu}{\sigma}\right)>0.
\end{align*}
Hence $\Exp(|Y_{\mu,\sigma^2}|)$ is a strictly increasing function in $\sigma>0$ from which,
 by Subsection \ref{subsection_gauus_cond_gauss} together with \eqref{compare_2}, we get for all $0<t<T$,
\begin{align*}
\Exp\left(|W_t-W_{t}^{\rm av}|\right)
  & =\Exp\left(|W_t-W_{t}^{\rm st}|\right)\\
  & =\frac{2t}{\sqrt{T}}\Phi'\left(\frac{b}{\sqrt{T}}\right)
     + b\frac{t}{T}\left(2\Phi\left(\frac{b}{\sqrt{T}}\right)-1\right)\\
  & > 2\sigma(t)\Phi'\left(\frac{bt}{T\sigma(t)}\right)
      + b\frac{t}{T}\left(2\Phi\left(\frac{bt}{T\sigma(t)}\right)-1\right)
    = \Exp\left(|W_t-W_{t}^{\rm ir}|\right)
\end{align*}
concluding the proof.
\end{proof}

Next we compare expected absolute path deviations
 \ $\Exp\left(\int_0^T|W_t-W_{t}^{\rm br}|\,\dd t\right)$.
\ Using that integration over the time-interval \ $[0,T]$ \ and taking expectation can be interchanged
 (as it is explained in the introduction), by Lemma \ref{prop3}, we also get
 \begin{align*}
   \Exp\left(\int_0^T|W_t-W_{t}^{\rm av}|\,\dd t\right)
       =\Exp\left(\int_0^T|W_t-W_{t}^{\rm st}|\,\dd t\right)
    >\Exp\left(\int_0^T|W_t-W_{t}^{\rm ir}|\,\dd t\right).
 \end{align*}

Using \eqref{help_exp_abs_path_dev2} and Proposition \ref{prop2},
 it might also be possible to calculate and to compare expected conditional absolute path deviations given
 $W_T=d$.
This task is more complicated, since now the mean is different for different versions of the bridge,
 see Proposition \ref{prop2}.
Instead we will now consider expected (conditional) quadratic path deviations which
 have much nicer forms.

Next we calculate the second moments \ $\Exp\left((W_t-W_{t}^{\rm br})^2\right)$ \
 of the path deviations \ $W_t-W_{t}^{\rm br}$, \ and also expected quadratic path deviations
 \ $\Exp\left(\int_0^T(W_t-W_{t}^{\rm br})^2\,\dd t\right)$.

\begin{thm}\label{prop4}
Let \ $(W_{t}^{\rm br})_{t\in[0,T]}$ \ be a Wiener bridge from \ $0$ \ to \ $b$ \
 over the time-interval \ $[0,T]$, \ where \ $b\in\RR$.
Then for all \ $t\in[0,T)$,
 \begin{align}\label{help_exp_quad_path_dev1}
  &\Exp\left((W_t-W_{t}^{\rm av})^2\right)=\Exp\left((W_t-W_{t}^{\rm st})^2\right)
     =\frac{t^2}{T}+b^2\frac{t^2}{T^2}\\ \label{help_exp_quad_path_dev2}
  & \Exp\left((W_t-W_{t}^{\rm ir})^2\right) = \sigma^2(t)+b^2\frac{t^2}{T^2},
 \end{align}
 where \ $\sigma^2(t)$ \ is defined in Proposition \ref{prop1}.
Moreover, the expected quadratic path deviations take the following forms:
 \begin{align*}
   & \Exp\left(\int_0^T(W_t-W_{t}^{\rm av})^2\,\dd t\right)
      = \Exp\left(\int_0^T(W_t-W_{t}^{\rm st})^2\,\dd t\right)
      = \frac{T}{3}(T+b^2),\\
   & \Exp\left(\int_0^T(W_t-W_{t}^{\rm ir})^2\,\dd t\right)
      = \frac{T}{3}\left(\frac{T}{2} + b^2\right).
 \end{align*}
\end{thm}

\begin{proof}
For a normally distributed random variable $Y_{\mu,\sigma^2}$ with mean $\mu$ and with variance $\sigma^2\geq0$ we clearly have
\begin{equation}\label{quaddist}
\Exp(Y_{\mu,\sigma^2}^2)=\Var(Y_{\mu,\sigma^2})+(\Exp(Y_{\mu,\sigma^2}))^2=\sigma^2+\mu^2.
\end{equation}
Hence, by Proposition \ref{prop1},
 we get \eqref{help_exp_quad_path_dev1} and \eqref{help_exp_quad_path_dev2}.
Then, we have
\begin{align*}
\Exp\left(\int_0^T(W_t-W_{t}^{\rm av})^2\,\dd t\right)
 & =\Exp\left(\int_0^T(W_t-W_{t}^{\rm st})^2\,\dd t\right)\\
 & =\int_0^T\frac{t^2}{T^2}(T+b^2)\,\dd t=\frac{T}{3}(T+b^2),
\end{align*}
and, by a change of variables $s=(T-t)/T$ and partial integration, we get
\begin{align*}
\Exp\left(\int_0^T(W_t-W_{t}^{\rm ir})^2\,\dd t\right)
 & = \int_0^T \left(\sigma^2(t) + b^2 \frac{t^2}{T^2}\right)\,\dd t\\
 & =\int_0^T \left[t\left(2-\frac{t}{T}\right)+2(T-t)\log\frac{T-t}{T}+b^2\frac{t^2}{T^2}\right]\,\dd t\\
 & =T^2-\frac13T^2+2T^2\int_0^1s\log s\,\dd s+\frac13b^2T\\
 & =\frac23T^2-T^2\int_0^1s\,\dd s+\frac13b^2T\\
 & =\frac16T^2+\frac13b^2T=\frac{T}{3}\left(\frac{T}{2}+b^2\right)
\end{align*}
concluding the proof.
\end{proof}

Note that, by Theorem \ref{prop4} and \eqref{compare_2}, for all \ $t\in(0,T)$,
 \[
  \Exp\left((W_t-W_{t}^{\rm av})^2\right)=\Exp\left((W_t-W_{t}^{\rm st})^2\right)
      >\sigma^2(t)+b^2\frac{t^2}{T^2}=\Exp\left((W_t-W_{t}^{\rm ir})^2\right).
 \]
Further, in case $b=0$ this shows that the expected quadratic path deviation of the integral representation
 is half of those of the anticipative version and the space-time transform of the bridge.
This is in accordance with the typical observations from simulation studies as in Figure \ref{Figure1}.

Next we study expected conditional quadratic path deviations.

\begin{thm}\label{prop5}
Let \ $(W_{t}^{\rm br})_{t\in[0,T]}$ \ be a Wiener bridge from \ $0$ \ to \ $b$ \
 over the time-interval \ $[0,T]$, \ where \ $b\in\RR$.
Then for all \ $t\in[0,T)$ \ and \ $d\in\RR$ \ we have
 \begin{align}\label{help_exp_cond_quad1}
   &\Exp\left((W_t-W_{t}^{\rm av})^2\,\big|\,W_T=d\right)=(d-b)^2\frac{t^2}{T^2},\\\label{help_exp_cond_quad2}
   &\Exp\left((W_t-W_{t}^{\rm ir})^2\,\big|\,W_T=d\right)
     = 2t\frac{T-t}{T}+2\frac{(T-t)^2}{T}\log\frac{T-t}{T} \\\nonumber
   & \phantom{\Exp\left((W_t-W_{t}^{\rm ir})^2\,\big|\,W_T=d\right)=}
     -\frac{(T-t)^2}{T}\left(\log\frac{T-t}{T}\right)^2\\\nonumber
   & \phantom{\Exp\left((W_t-W_{t}^{\rm ir})^2\,\big|\,W_T=d\right)=}
     +\left((d-b)\frac{t}{T}+d\frac{T-t}{T}\log\frac{T-t}{T}\right)^2,\\ \label{help_exp_cond_quad3}
   &\Exp\left((W_t-W_{t}^{\rm st})^2\,\big|\,W_T=d\right)
     = \frac{t^2}{T}-\frac{(2t-T)^2}{T} 1_{[\frac{T}{2},T)}(t)\\\nonumber
   &\phantom{\Exp\left((W_t-W_{t}^{\rm st})^2\,\big|\,W_T=d\right) = }
      + \left(b\frac{t}{T}-d\frac{(2t-T)}T 1_{[\frac{T}{2},T)}(t)\right)^2.
 \end{align}
Moreover,
 the expected conditional quadratic path deviations take the following forms:
 \begin{align}\label{help_exp_cond_quad4}
  & \Exp\left(\int_0^T(W_t-W_{t}^{\rm av})^2\,\dd t\,\bigg|\,W_T=d\right)=\frac13(d-b)^2T,\\
  & \Exp\left(\int_0^T(W_t-W_{t}^{\rm ir})^2\,\dd t\,\bigg|\,W_T=d\right)   \label{help_exp_cond_quad5}
     = \frac7{54}(b-d)^2T+\frac{11}{54}b^2T-\frac{7}{54}dbT+\frac1{27}T^2,\\   \label{help_exp_cond_quad6}
 & \Exp\left(\int_0^T(W_t-W_{t}^{\rm st})^2\,\dd t\,\bigg|\,W_T=d\right)
     = \frac16(d-b)^2T+\frac16b^2T-\frac1{12}dbT+\frac16T^2.
 \end{align}
\end{thm}

\begin{proof}
By \eqref{help_cond_gauss1}, \eqref{help_cond_gauss4} and \eqref{quaddist}, for $0<t<T$ we get
 \eqref{help_exp_cond_quad1}.
Using that integration over the time-interval \ $[0,T]$ \ and taking conditional
 expectation can be interchanged (as explained in the Introduction),  we get
 \eqref{help_exp_cond_quad1} yields \eqref{help_exp_cond_quad4}.
By \eqref{help_cond_gauss2}, \eqref{help_cond_gauss5} and \eqref{quaddist},
 we have \eqref{help_exp_cond_quad2}, and hence, by a change of variables $s=(T-t)/T$ and partial integration,
 we get
\begin{align*}
& \Exp\left(\int_0^T(W_t-W_{t}^{\rm ir})^2\,\dd t\,\bigg|\,W_T=d\right)\\
& \quad=\int_0^T\Bigg[2t\frac{T-t}{T}+(d-b)^2\frac{t^2}{T^2}+2d(d-b)\frac{t}{T}\frac{T-t}{T}\log\frac{T-t}{T}\\
& \quad\phantom{=\int_0^T\quad }
      + 2T\frac{(T-t)^2}{T^2}\log\frac{T-t}{T}
      + (d^2-T)\frac{(T-t)^2}{T^2}\left(\log\frac{T-t}{T}\right)^2\Bigg]\,\dd t\\
& \quad=\frac13(d-b)^2T+\int_0^1\Big[2T^2(1-s)s+2d(d-b)T(1-s)s\log s\\
& \quad\phantom{=\frac13(d-b)^2T+\int_0^1}+2T^2s^2\log s+(d^2-T)Ts^2 (\log s)^2\Big]\,\dd s\\
& \quad=\frac13(d-b)^2T+T^2-\frac23T^2-d(d-b)T\int_0^1s\,ds+\frac23d(d-b)T\int_0^1s^2\,\dd s\\
& \quad\phantom{=}-\frac23T^2\int_0^1s^2\,ds-(d^2-T)T\int_0^1\frac23s^2\log s\,\dd s\\
& \quad=\frac13(d-b)^2T+\frac13T^2-\frac12d(d-b)T+\frac29d(d-b)T-\frac29T^2+\frac29(d^2-T)T\int_0^1s^2\,\dd s\\
& \quad=\frac13(d-b)^2T+\frac19T^2-\frac5{18}d(d-b)T+\frac2{27}(d^2-T)T\\
& \quad=\frac1{27}T^2+\frac7{54}d^2T-\frac7{18}dbT+\frac13b^2T,
\end{align*}
 which yields \eqref{help_exp_cond_quad5}.
Finally, by \eqref{help_cond_gauss3}, \eqref{help_cond_gauss6} and \eqref{quaddist},
 we have \eqref{help_exp_cond_quad3}, and hence, by a change of variables $s=2t-T$, we get
\begin{align*}
& \Exp\left(\int_0^T(W_t-W_{t}^{\rm st})^2\,\dd t\,\bigg|\,W_T=d\right)\\
& \quad=\frac13T^2+\frac13b^2T
     -\int_{\frac{T}{2}}^T\Bigg[\frac{(2t-T)^2}{T}+\frac{2dbt(2t-T)}{T^2}-\frac{d^2}{T^2}(2t-T)^2\Bigg]\,\dd t\\
& \quad=\frac13T^2+\frac13b^2T
         -\frac12\int_0^T\Bigg[\frac{s^2}{T}+\frac{db(s+T)s}{T^2}-\frac{d^2}{T^2}s^2\Bigg]\,\dd s\\
& \quad=\frac13T^2+\frac13b^2T-\frac12\left(\frac13T^2+\frac13dbT+\frac12dbT-\frac13d^2T\right)\\
& \quad=\frac16T^2+\frac13b^2T-\frac5{12}dbT+\frac16d^2T,
\end{align*}
 which yields \eqref{help_exp_cond_quad6}.
\end{proof}

In what follows we give a complete comparison of the quantities
 \eqref{help_exp_cond_quad4}, \eqref{help_exp_cond_quad5} and \eqref{help_exp_cond_quad6}.
Let $\tilde b=b/\sqrt{T}$ and $\tilde d=d/\sqrt{T}$.
 Using the notation
$$
  e_{\rm br}:=\Exp\left(\int_0^T(W_t-W_{t}^{\rm br})^2\,\dd t\,\bigg|\,W_T=d\right),
$$
 by Theorem \ref{prop5}, we have
\begin{align*}
e_{\rm av} & =\frac13(\tilde b-\tilde d)^2T^2,\\
e_{\rm ir} & =\left(\frac7{54}(\tilde b-\tilde d)^2
  +\frac{11}{54}\tilde b^2-\frac{7}{54}\tilde b\tilde d+\frac1{27}\right)T^2,\\
e_{\rm st} & =\left(\frac1{6}(\tilde b-\tilde d)^2+\frac{1}{6}\tilde b^2
               -\frac{1}{12}\tilde b\tilde d+\frac1{6}\right)T^2.
\end{align*}
Hence we easily calculate
\begin{align*}
 e_{\rm av}>e_{\rm ir}
 & \iff\frac{11}{54}(\tilde b-\tilde d)^2>\frac{11}{54}\tilde b^2-\frac7{54}\tilde b \tilde d+\frac1{27}\\[2mm]
 & \iff\left|\tilde d-\frac{15}{22}\tilde b\right|>\sqrt{\frac2{11}
     +\left(\frac{15}{22}\right)^2(\tilde b)^2},\\[2ex]
 e_{\rm av}>e_{\rm st} & \iff\frac{1}{6}(\tilde b-\tilde d)^2
    >\frac{1}{6}\tilde b^2 -\frac1{12}\tilde b\tilde d+\frac1{6}\\[2mm]
 & \iff\left|\tilde d-\frac{3}{4}\tilde b\right|>\sqrt{1+\frac{9}{16}(\tilde b)^2}
   \intertext{and}
 e_{\rm st} < e_{\rm ir} & \iff\frac{1}{27}(\tilde b-\tilde d)^2<\frac{1}{27}\tilde b^2
      -\frac5{108}\tilde b\tilde d-\frac7{54}\\[2mm]
 & \iff
    \left|\tilde d-\frac{3}{8}\tilde b\right|
       <\sqrt{\frac{9}{64}(\tilde b)^2-\frac72}\quad\text{ and }\quad(\tilde b)^2\geq \frac{224}{9}.
 \end{align*}
The corresponding regions are graphically illustrated in Figure \ref{Figure3}.
\begin{figure}[h]
\includegraphics[scale=0.45]{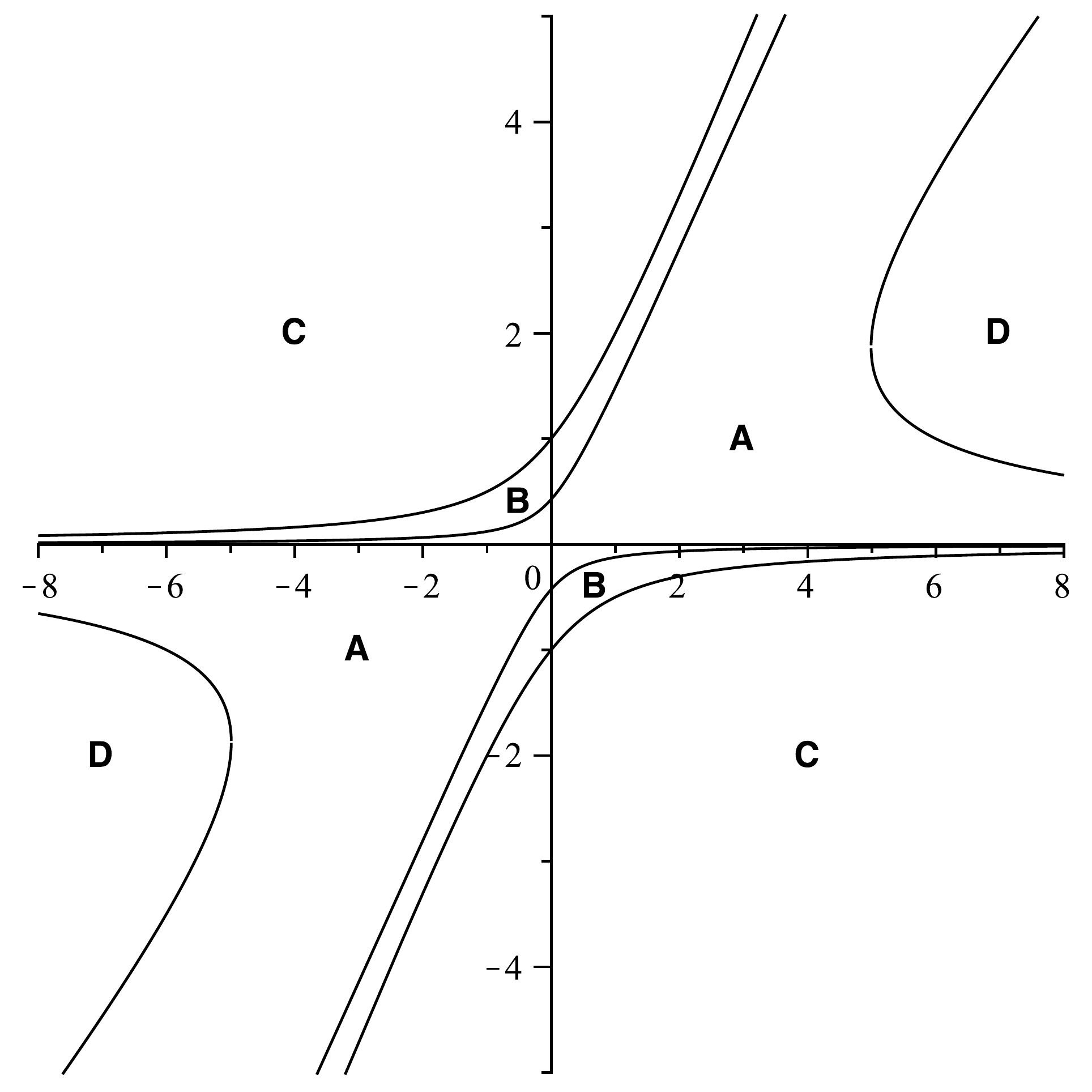}
\caption{\small Regions in the $(\tilde b,\tilde d)$-plain for which {\rm\bf A}:
 $e_{\rm av}<e_{\rm ir}<e_{\rm st}$, {\rm\bf B}: $e_{\rm ir}<e_{\rm av}<e_{\rm st}$,
 {\rm\bf C}: $e_{\rm ir}<e_{\rm st}<e_{\rm av}$, and {\rm\bf D}: $e_{\rm av}<e_{\rm st}<e_{\rm ir}$.}
\label{Figure3}
\end{figure}

 Finally, we remark that Theorem \ref{prop5} justifies our simulation results in the case of the endpoint
 \ $W_T$ \ of the Wiener sample path is close to the prescribed endpoint \ $b$ \ of its
 bridge.
Indeed, in case of \ $d=b$, \ by Theorem \ref{prop5}, we get
 \begin{align*}
   & \Exp\left(\int_0^T(W_t-W_{t}^{\rm av})^2\,\dd t\,\bigg|\,W_T=d\right) = 0,\\
   & \Exp\left(\int_0^T(W_t-W_{t}^{\rm ir})^2\,\dd t\,\bigg|\,W_T=d\right)
          = \frac{2}{27}d^2T + \frac{1}{27}T^2,\\
   & \Exp\left(\int_0^T(W_t-W_{t}^{\rm st})^2\,\dd t\,\bigg|\,W_T=d\right) = \frac{1}{12}d^2 T+\frac{1}{6}T^2,
 \end{align*}
 which shows that the expected conditional quadratic path deviation of the Wiener process
 from the anticipative version of its bridge is \ $0$ \ being smaller than from
 the integral representation of the bridge or from the space-time bridge version.

\section{Path deviation of the Ornstein-Uhlenbeck process from its bridges}

\subsection{Ornstein-Uhlenbeck bridge versions}

Let $(U^{a}_{t})_{t\geq0}$ be a one-dimensional Ornstein-Uhlenbeck process starting in $a\in\RR$,
 i.e., it is the unique strong solution of the SDE
$$
 \dd U^{a}_{t}=q\,U^{a}_{t}\,\dd t+\sigma\,\dd W_{t}
  \quad\text{ with initial condition}\quad U^{a}_{0}=a
$$
for some $q\not=0$ and $\sigma>0$, where $(W_t)_{t\geq 0}$ is a standard Wiener process.
It is well-known that the Ornstein-Uhlenbeck process has the integral representation
$$
  U^{a}_{t}=\ee^{qt}\Big(a+\sigma\int_{0}^t\ee^{-qs}\,\dd W_{s}\Big),\qquad t\geq0,
$$
 which is a Gauss process with mean function $\Exp(U^{a}_{t})=a\ee^{qt}$ and
 covariance function $\Cov(U^{a}_{s},U^{a}_{t})=\sigma^2\frac{\ee^{qt}}{q}\,\sinh(qs)$ for $0\leq s\leq t$.
We also have \ $U^{a}_t = a\ee^{qt} + U^{0}_{t}$, $t\geq 0$, \ where \ $(U^{0}_t)_{t\geq 0}$ \ is a
one-dimensional Ornstein-Uhlenbeck process starting in 0.

We consider the following versions of the Ornstein-Uhlenbeck bridge from $a$ to $b$ over the
 time-interval $[0,T]$, where $a,b\in\RR$:

{\bf  1. Anticipative version}
$$
 U_{t}^{{\rm av}}
  = a\,\frac{\sinh(q(T-t))}{\sinh(qT)}
    + b\,\frac{\sinh(qt)}{\sinh(qT)}
    + \left(U^{0}_{t}-\frac{\sinh(qt)}{\sinh(qT)}\,U^{0}_{T}\right)\,,\, 0\leq t\leq T.
 $$
Up to our knowledge this anticipative version of the Ornstein-Uhlenbeck bridge first appears
 on page 378 of Donati-Martin \cite{DonMar} for $a=b=0$ and in Lemma 1 of Papie\.{z} and
 Sandison \cite{PapSan} for special values of $q$ and $\sigma$.
It is also an easy consequence of Theorem 2 in Delyon and Hu \cite{DelHu} and of Proposition 4 in Gasbarra,
 Sottinen and Valkeila \cite{GasSotVal}.

{\bf  2. Integral representation}
$$
  U_{t}^{{\rm ir}}=
   \begin{cases}
     \displaystyle a\,\frac{\sinh(q(T-t))}{\sinh(qT)}+b\,\frac{\sinh(qt)}{\sinh(qT)}
            +\sigma\int_{0}^t\frac{\sinh(q(T-t))}{\sinh(q(T-s))}\,\dd W_{s} & \text{if \ $0\leq t<T$,}\\
      b & \text{if \ $t=T$.}
    \end{cases}
$$
This integral representation of the Ornstein-Uhlenbeck bridge is the unique strong solution of the below
 given SDE \eqref{oubridgesde}, see, e.g., Barczy and Kern \cite[Remark 3.9]{Barpbk}.

{\bf 3. Space-time transform}
 $$
 U_{t}^{{\rm st}}=
   \begin{cases}
     \displaystyle a\,\frac{\sinh(q(T-t))}{\sinh(qT)}+b\,\frac{\sinh(qt)}{\sinh(qT)}
       +\sigma\,\ee^{qt}\,\frac{\kappa(T)-\kappa(t)}{\kappa(T)}\,W_{\frac{\kappa(t)\kappa(T)}{\kappa(T)-\kappa(t)}} & \text{if \ $0\leq t<T$,}\\
     b & \text{if \ $t=T$,}
   \end{cases}
 $$
with the strictly increasing time-change
$$
   \RR\ni t\mapsto\kappa(t)=\frac{\ee^{-qt}\sinh(qt)}{q}=\frac{1-\ee^{-2qt}}{2q}.
$$
This space-time transform of the Ornstein-Uhlenbeck bridge goes back to the proof of Lemma 1
 in Papie\.{z} and Sandison \cite{PapSan} and is roughly speaking a time-transformation by $\kappa$ and a
 rescaling with the coefficient \ $\ee^{qt}$ \ of the space-time transform representation
 $(W_{t}^{{\rm st}})_{t\in[0,T]}$ of the Wiener bridge from $a$ to $b$ over the time-interval $[0,T]$.

\begin{remark}
We note that the previous versions of an Ornstein-Uhlenbeck bridge are in accordance with the
 corresponding versions of a usual standard Wiener bridge introduced in the introduction.
By this we mean that for all \ $T>0$, \ $t\in[0,T]$ \ and \ $\sigma=1$, \ $U_{t}^{\rm br}$
 \ converges to \ $W_{t}^{\rm br}$ \ in \ $L^2(\Omega,\mathcal F,\PP)$ \ as \ $q\to0$.
\ Indeed,
 \begin{align*}
   &\lim_{q\to 0} \frac{\sinh(q(T-t))}{\sinh(qT)}
       = \lim_{q\to 0} \frac{(T-t)\cosh(q(T-t))}{T\cosh(qT)}
       = \frac{T-t}{T}, \\[1mm]
  &\lim_{q\to 0} \frac{\sinh(qt)}{\sinh(qT)}
       = \lim_{q\to 0} \frac{t\cosh(qt)}{T\cosh(qT)}
       = \frac{t}{T},\\[1mm]
  &\lim_{q\to 0}
    \ee^{qt}\frac{\kappa(T)-\kappa(t)}{\kappa(T)}
      = \lim_{q\to 0} \frac{\ee^{-qt} - \ee^{q(t-2T)}}{1 - \ee^{-2qT} }
      = \lim_{q\to 0} \frac{-t\ee^{-qt} - (t-2T)\ee^{q(t-2T)}}{2T\ee^{-2qT}}
      = \frac{T-t}{T},\\[1mm]
  &\lim_{q\to 0}\frac{\kappa(t)\kappa(T)}{\kappa(T)-\kappa(t)}
     = \lim_{q\to 0}\left( \kappa(t) + \frac{\kappa(t)^2}{\kappa(T) - \kappa(t) }\right)
     =  t +  \lim_{q\to 0} \frac{(1- \ee^{-2qt})^2}{2q(\ee^{-2qt} - \ee^{-2qT})} \\
  &\phantom{\lim_{q\to 0}\frac{\kappa(t)\kappa(T)}{\kappa(T)-\kappa(t)}\;}
     = t + \frac{t^2}{T-t}
     = \frac{tT}{T-t}.
 \end{align*}
Further,
 \begin{align*}
   \Exp(U_t^0-W_t)^2
   & = \Exp\left( \int_0^t (\ee^{q(t-s)} - 1)\,\dd W_s\right)^2
     = \int_0^t (\ee^{q(t-s)} - 1)^2 \,\dd s \\
   & = \frac{\ee^{2qt}-1}{2q} + \frac{2}{q}(1-\ee^{qt}) + t
        \to 0 \qquad \text{as \ $q\to 0$,}
 \end{align*}
 and
 \begin{align*}
  &\Exp\left(\int_0^t\frac{\sinh(q(T-t))}{\sinh(q(T-s))}\,\dd W_s
              - \int_0^t \frac{T-t}{T-s}\,\dd W_s \right)^2 \\
  &\qquad\qquad = \int_0^t \left( \frac{\sinh(q(T-t))}{\sinh(q(T-s))}
              - \frac{T-t}{T-s}\right)^2\,\dd s
     \to 0  \qquad \text{as \ $q\to 0$,}
 \end{align*}
 where the convergence follows by Lebesgue's dominated convergence theorem.
\proofend
\end{remark}

In all what follows we will use the notation $(U_{t}^{\rm br})_{t\in[0,T]}$ if the version of the bridge
 is not specified.

First we present a lemma about a time-transformation which will be useful for calculating
 $\Var(U^{a}_{t}-U_{t}^{\rm st})$ and $\Cov(U_{s}^{\rm st},U_{t}^{\rm st})$, $0\leq s,t<T$.

\begin{lemma}\label{lemma_transform}
For the time-transformation \ $\kappa_{T}^\ast(t):=\frac{\kappa(t)\kappa(T)}{\kappa(T)-\kappa(t)}$,
 $t\in[0,T)$, with $\kappa(t):=\frac{1-\ee^{-2qt}}{2q}$, $t\in\RR$, \
 we get \ $\kappa_T^*$ \ is strictly increasing and \ $\kappa_{T}^\ast(t)\geq t$ \ for all \ $t\in[0,T)$.
\end{lemma}

\begin{proof}
Since the function $[0,T)\ni t\mapsto\frac{tT}{T-t}$ is strictly increasing and
 $\RR\ni t\mapsto\kappa(t)=\frac{1-\ee^{-2qt}}{2q}$ is strictly increasing for every $q\not=0$,
 we get that $[0,T)\ni t\mapsto\frac{\kappa(t)\kappa(T)}{\kappa(T)-\kappa(t)}=:\kappa_{T}^\ast(t)$
 is strictly increasing.
Further, easy calculations show that
\begin{enumerate}
\item $\kappa_{T}^\ast(0)=0$ and $\lim_{t\uparrow T}\kappa_{T}^\ast(t)=\infty$.
\item $\kappa_{T}^\ast$ is differentiable on $[0,T)$, namely
\begin{align*}
(\kappa_{T}^\ast)'(t) &
 =\frac{\kappa'(t)\kappa(T)(\kappa(T)-\kappa(t))+\kappa(t)\kappa(T)\kappa'(t)}{(\kappa(T)-\kappa(t))^2}
 =\frac{\kappa'(t)\kappa^2(T)}{(\kappa(T)-\kappa(t))^2},
 \quad t\in[0,T),
\end{align*}
with $\kappa'(t)=\ee^{-2qt}$ and hence $(\kappa_{T}^\ast)'(0)=1$.
\item For the second derivative we get
\begin{align*}
(\kappa_{T}^\ast)''(t)
& =\frac{\kappa''(t)\kappa^2(T)(\kappa(T)-\kappa(t))^2
    +2(\kappa(T)-\kappa(t))(\kappa')^2(t)\kappa^2(T)}{(\kappa(T)-\kappa(t))^4}\\
& =\frac{\kappa^2(T)\big(\kappa''(t)(\kappa(T)-\kappa(t))+2(\kappa')^2(t)\big)}{(\kappa(T)-\kappa(t))^3},
 \quad t\in[0,T).
\end{align*}
Since $\kappa''(t)(\kappa(T)-\kappa(t))+2(\kappa')^2(t)=\ee^{-4qt}+\ee^{-2q(T+t)}>0$ we have
 $(\kappa_{T}^\ast)'$ is strictly increasing.
\end{enumerate}
Altogether this shows that $(\kappa_{T}^\ast)'(t)\geq (\kappa_{T}^\ast)'(0)=1$, $t\in[0,T)$ and hence
 $\kappa_{T}^\ast(t)\geq t$ for all $t\in[0,T)$.
\end{proof}

\begin{prop}\label{prop6}
Let \ $(U_{t}^{\rm br})_{t\in[0,T]}$ \ be an Ornstein-Uhlenbeck bridge from \ $a$ \ to \ $b$ \
 over the time-interval \ $[0,T]$, \ where \ $a,b\in\RR$.
Then \ $(U_{t}^{\rm br})_{t\in[0,T]}$ \ is a Gauss process with mean function
 \[
   \Exp(U_t^{{\rm br}})=a\,\frac{\sinh(q(T-t))}{\sinh(qT)}+b\,\frac{\sinh(qt)}{\sinh(qT)},\qquad 0\leq t<T,
 \]
 and with covariance function
 \begin{equation}\label{oubridgecov}
  \Cov(U_{s}^{\rm br},U_{t}^{\rm br})=\frac{\sigma^2}{q}\,\frac{\sinh(qs)\sinh(q(T-t))}{\sinh(qT)},
   \qquad  0\leq s\leq t<T.
\end{equation}
Hence all the bridge versions above have the same finite-dimensional distributions.
\end{prop}

\begin{proof}
For $0\leq s\leq t<T$, we have the covariance function
 \begin{align*}
  \Cov(U_{s}^{\rm av},U_{t}^{\rm av})
    & = \Cov\left(U^{0}_{s}-\frac{\sinh(qs)}{\sinh(qT)}\,U^{0}_{T},
                  U^{0}_{t}-\frac{\sinh(qt)}{\sinh(qT)}\,U^{0}_{T}\right)\\
    & = \Cov(U^{0}_{s},U^{0}_{t})-\frac{\sinh(qt)}{\sinh(qT)}
        \,\Cov(U^{0}_{s},U^{0}_{T})-\frac{\sinh(qs)}{\sinh(qT)}\,\Cov(U^{0}_{t},U^{0}_{T})\\
    &   \phantom{=\;\;\;}+\frac{\sinh(qs)\sinh(qt)}{\sinh^2(qT)}\,\Cov(U^{0}_{T},U^{0}_{T})\\
    & = \sigma^2\frac{\ee^{qt}}{q}\,\sinh(qs)-\frac{\sinh(qt)}{\sinh(qT)}\,\sigma^2\frac{\ee^{qT}}{q}\,\sinh(qs)\\
    & = \frac{\sigma^2}{q}\,\frac{\sinh(qs)\sinh(q(T-t))}{\sinh(qT)}
\end{align*}
and
\begin{align*}
\Cov(U_{s}^{\rm ir},U_{t}^{\rm ir})
 & =\sigma^2\Cov\left(\int_{0}^s\frac{\sinh(q(T-s))}{\sinh(q(T-r))}\,\dd W_{r},
                        \int_{0}^t\frac{\sinh(q(T-t))}{\sinh(q(T-r))}\,\dd W_{r}\right)\\
& =\sigma^2\int_{0}^s\frac{\sinh(q(T-s))\sinh(q(T-t))}{\sinh^2(q(T-r))}\,\dd r\\
& =\frac{\sigma^2}{q}\,\sinh(q(T-s))\sinh(q(T-t))\int_{q(T-s)}^{qT}\frac1{\sinh^2v}\,\dd v\\
& =\frac{\sigma^2}{q}\,\sinh(q(T-s))\sinh(q(T-t))\left(\frac{\cosh(q(T-s))}{\sinh(q(T-s))}
           -\frac{\cosh(qT)}{\sinh(qT)}\right)\\
& =\frac{\sigma^2}{q}\,\sinh(q(T-s))\sinh(q(T-t))\,\frac{\sinh(qT-q(T-s))}{\sinh(q(T-s))\sinh(qT)}\\
& =\frac{\sigma^2}{q}\,\frac{\sinh(qs)\sinh(q(T-t))}{\sinh(qT)}.
\end{align*}
By Lemma \ref{lemma_transform}, for $0\leq s\leq t<T$ we get
\begin{align*}
 \Cov(U_{s}^{\rm st},U_{t}^{\rm st})
  & =\sigma^2\ee^{q(s+t)}\,\frac{\kappa(T)-\kappa(s)}{\kappa(T)}\,\frac{\kappa(T)-\kappa(t)}{\kappa(T)}
      \,\Cov\left(W_{\frac{\kappa(s)\kappa(T)}{\kappa(T)-\kappa(s)}},
                  W_{\frac{\kappa(t)\kappa(T)}{\kappa(T)-\kappa(t)}}\right)\\
& =\sigma^2\ee^{q(s+t)}\,\frac{\kappa(T)-\kappa(s)}{\kappa(T)}
    \,\frac{\kappa(T)-\kappa(t)}{\kappa(T)}\,\frac{\kappa(s)\kappa(T)}{\kappa(T)-\kappa(s)}\\
& = \sigma^2\ee^{qt}\frac{\sinh(qs)}{\ee^{-qT}\sinh(qT)}\frac{\ee^{-2qt}-\ee^{-2qT}}{2q}
  =\frac{\sigma^2}{q}\,\frac{\sinh(qs)\sinh(q(T-t))}{\sinh(qT)}
\end{align*}
concluding the proof.
\end{proof}

In what follows we study the continuity of the sample paths of the bridge versions.
It follows from the definitions that all bridge versions have almost sure continuous sample paths
 on \ $[0,T)$.
\ The (left) continuity of the trajectories at \ $t=T$ \ is also obvious in case of the anticipative version,
 but not in case of the integral representation and the space-time transform.
The strong law of large numbers for a standard Wiener process (see, e.g., Problem 2.9.3 in Karatzas and Shreve
 \cite{KarShr}) yields the desired continuity for the space-time transform.
The above mentioned continuity for the integral representation follows from
 Lemma 4.5 in Barczy and Kern \cite{Barpbk}.
For the sake of completeness and easier lucidity, in the Appendix we formulate and prove this lemma in the
 present setting (without reference to the notations in Barczy and Kern \cite{Barpbk}).

Hence the anticipative version \ $U^{\rm av}$, \ the integral representation \ $U^{\rm ir}$ \ and the
 space-time transform \ $U^{\rm st}$ \ induce the same probability measure on
 \ $(C[0,T],\cB(C[0,T]))$. \
This underlines and explains the definition of an Ornstein-Uhlenbeck bridge from \ $a$ \ to \ $b$
 \ over the time-interval \ $[0,T]$, \ by which we mean any almost surely continuous Gauss process
 having mean function and covariance function given in Proposition \ref{prop6}.

 We also note that the finite dimensional distributions of the Ornstein-Uhlenbeck bridge
 versions coincide with the conditional finite dimensional distributions of the Ornstein-Uhlenbeck
 process \ $(U^{a}_{t})_{t\in[0,T]}$ \ (starting in $a$) and conditioned on \ $\{U^{a}_{T}=b\}$, \ see, e.g.,
 Delyon and Hu \cite[Theorem 2]{DelHu}, Gasbarra, Sottinen and Valkeila
 \cite[Proposition 4]{GasSotVal} or Barczy and Kern \cite[Proposition 3.5]{Barpbk}.

\subsection{Different sample path behavior of Ornstein-Uhlenbeck bridge versions}

First we present an indicator for different sample path behavior of the Ornstein-Uhlenbeck
 bridge versions.
If we consider the linear SDE
\begin{equation}\label{oubridgesde}
 \dd U_{t}^{\rm br} = q\left(-\coth(q(T-t))U_{t}^{\rm br} + \frac{b}{\sinh(q(T-t))}\right)\dd t
                             + \sigma\,\dd W_{t} \,,\quad 0\leq t<T
\end{equation}
with initial condition \ $U_{0}^{\rm br}=a$, \ then the integral representation is the unique strong
 solution of this SDE (see, e.g., Delyon and Hu \cite[Proposition 3]{DelHu}
 or Barczy and Kern \cite[Remark 3.10]{Barpbk}) and the anticipative version and the space-time transform
 are only weak solutions.
Indeed, if the anticipative version and the space-time transform were also strong solutions, then,
 by the definition of strong solution, we would get
 \ $\PP(U_{t}^{\rm av}=U_{t}^{\rm ir},\; \forall\; t\in[0,T))=1$
 \ and $\PP(U_{t}^{\rm st}=U_{t}^{\rm ir},\; \forall\; t\in[0,T))=1$, which are not true.
 For example, it does not hold that \ $\Var (U^{a}_{t} - U_{t}^{\rm av}) = \Var (U^{a}_{t} - U_{t}^{\rm st})$
 \ for all \ $t\in[0,T)$, \ and \ $\Var (U^{a}_{t} - U_{t}^{\rm st}) = \Var (U^{a}_{t} - U_{t}^{\rm ir})$
 \ for all \ $t\in[0,T)$, \ by Propositions \ref{prop7} and \ref{prop8} below.
We present another indicator for different sample path behavior of the
 Ornstein-Uhlenbeck bridge versions by calculating the covariances
  \ $\Cov(U_{t}^{\rm br},U^{a}_{t})$ \ of the coordinates of the two-dimensional Gauss process
 \ $(U_{t}^{\rm br},U^{a}_{t})_{t\in[0,T]}$. Note that these formulas are hard to compare.

 \begin{prop}\label{prop11_OU_cov}
 For all $0<t<T$ and $a,b\in\RR$ we have
  \begin{align*}
 \Cov(U_t^{\rm av},U^{a}_{t}) & =\frac{\sigma^2}{q}\,\frac{\sinh(q(T-t))\sinh(qt)}{\sinh(qT)}\,,\\
 \Cov(U_t^{\rm ir},U^{a}_{t}) & =\frac{\sigma^2}{q}\,\ee^{-q(T-t)}\sinh(q(T-t))\left(qt+\log\left(\frac{\sinh(qT)}{\sinh(q(T-t))}\right)\right),\\
 \Cov(U_t^{\rm st},U^{a}_{t}) & =\frac{\sigma^2}{q}(\ee^{qt}-1)\frac{\sinh(q(T-t))}{\sinh(qT)}\,.
 \end{align*}
 \end{prop}

 \begin{proof}
 For the anticipative version we have
  \begin{align*}
 \Cov(U_t^{\rm av},U^{a}_t) & = \Cov\left(U^{0}_t-\frac{\sinh(qt)}{\sinh(qT)}U^{0}_T,U^{0}_t\right)\\
 & = \sigma^2\frac{\ee^{qt}}{q}\sinh(qt)-\frac{\sinh(qt)}{\sinh(qT)}\sigma^2\frac{\ee^{qT}}{q}\sinh(qt)\\
 & = \frac{\sigma^2}{q}\,\frac{\sinh(qt)}{\sinh(qT)}\left(\ee^{qt}\sinh(qT)-\ee^{qT}\sinh(qt)\right)\\
 & = \frac{\sigma^2}{q}\,\frac{\sinh(q(T-t))\sinh(qt)}{\sinh(qT)}.
 \end{align*}
 Using Lemma \ref{lemma_transform} we get for the space-time transform
 \begin{align*}
 \Cov(U_t^{\rm st},U^{a}_t) & = \Cov\left(\sigma \ee^{qt}\frac{\kappa(T)-\kappa(t)}{\kappa(T)}W_{\kappa_T^\ast(t)},
                                                   U^{0}_t\right)\\
 & = \sigma^2\ee^{2qt}\frac{\kappa(T)-\kappa(t)}{\kappa(T)}\Cov\left(\int_0 ^{\kappa_T^\ast(t)}\dd W_s,\int_0 ^t\ee^{-qs}\,\dd W_s\right)\\
 & = \sigma^2\ee^{2qt}\frac{\kappa(T)-\kappa(t)}{\kappa(T)}\int_0 ^t\ee^{-qs}\,\dd s\\
 & = \sigma^2\ee^{2qt}\frac{\kappa(T)-\kappa(t)}{\kappa(T)}\frac{1-\ee^{-qt}}{q}\\
 & = \frac{\sigma^2}{q}(\ee^{qt}-1)\frac{\sinh(q(T-t))}{\sinh(qT)}\,.
 \end{align*}
 Finally, we get for the integral representation
 \begin{align*}
 \Cov(U_t^{\rm ir},U^{a}_t) & = \Cov\left(\sigma\int_0^t\frac{\sinh(q(T-t))}{\sinh(q(T-s))}\,\dd W_s,\sigma\int_0^t\ee^{q(t-s)}\,\dd W_s\right)\\
 & = \sigma^2\ee^{qt}\sinh(q(T-t))\int_0^t\frac{\ee^{-qs}}{\sinh(q(T-s))}\,\dd s\\
 & = 2\sigma^2\ee^{-q(T-t)}\sinh(q(T-t))\int_0^t\frac{1}{1-\ee^{-2q(T-s)}}\,\dd s\\
 & = \frac{\sigma^2}{q}\ee^{-q(T-t)}\sinh(q(T-t))\int_{\ee^{-2qT}}^{\ee^{-2q(T-t)}}\frac{1}{r(1-r)}\,\dd r\\
 & = \frac{\sigma^2}{q}\ee^{-q(T-t)}\sinh(q(T-t))\left(\log\left(\frac{\ee^{-2q(T-t)}}{\ee^{-2qT}}\right)
                                                       -\log\left(\frac{1-\ee^{-2q(T-t)}}{1- \ee^{-2qT} }\right)\right)\\
& =\frac{\sigma^2}{q}\,\ee^{-q(T-t)}\sinh(q(T-t))\left(qt+\log\left(\frac{\sinh(qT)}{\sinh(q(T-t))}\right)\right),
 \end{align*}
 concluding the proof.
 \end{proof}

Proposition \ref{prop11_OU_cov} also shows that, even though the three bridge versions have the same law
  on $(C[0,T],\cB(C[0,T]))$, their joint laws together with the Ornstein-Uhlenbeck process $U^{a}$ are different.

Our aim is to analyze the sample path deviations of the Ornstein-Uhlenbeck bridge versions to the original Ornstein-Uhlenbeck
 process $(U^{a}_t)_{t\in[0,T)}$ (starting in $a$) by calculating and comparing expected quadratic
 path deviations $\Exp\left(\int_0^T (U^{a}_t-U_{t}^{\rm br})^2\,\dd t\right).$

Simulation studies show the same qualitative behavior of typical sample path deviations of the anticipative version, the integral representation and the space-time transform of the Ornstein-Uhlenbeck bridge as we have
for the Wiener bridge, see the upper row of Figure \ref{Figure4}.
\begin{figure}[h]
\includegraphics[scale=0.25]{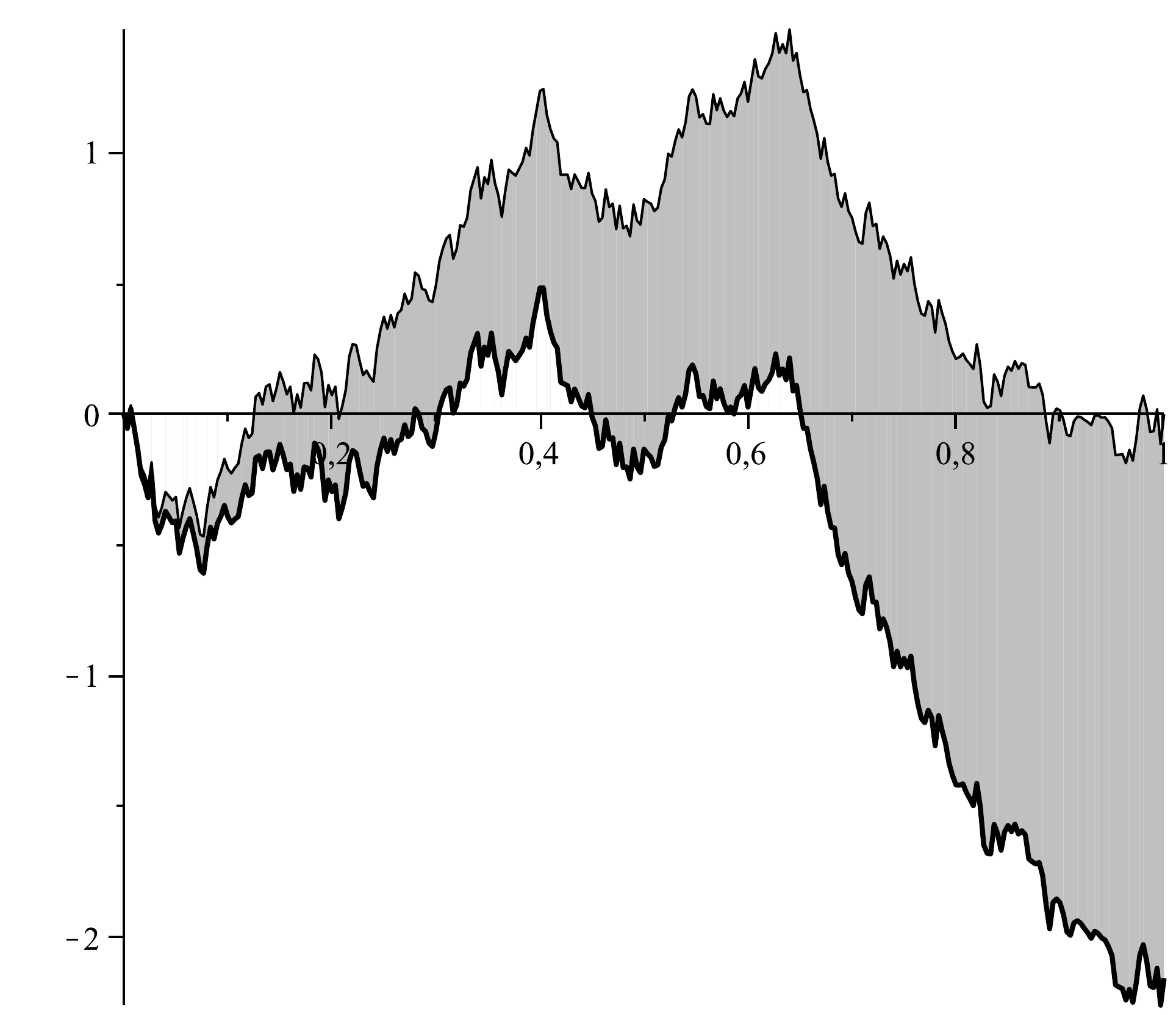}
\includegraphics[scale=0.25]{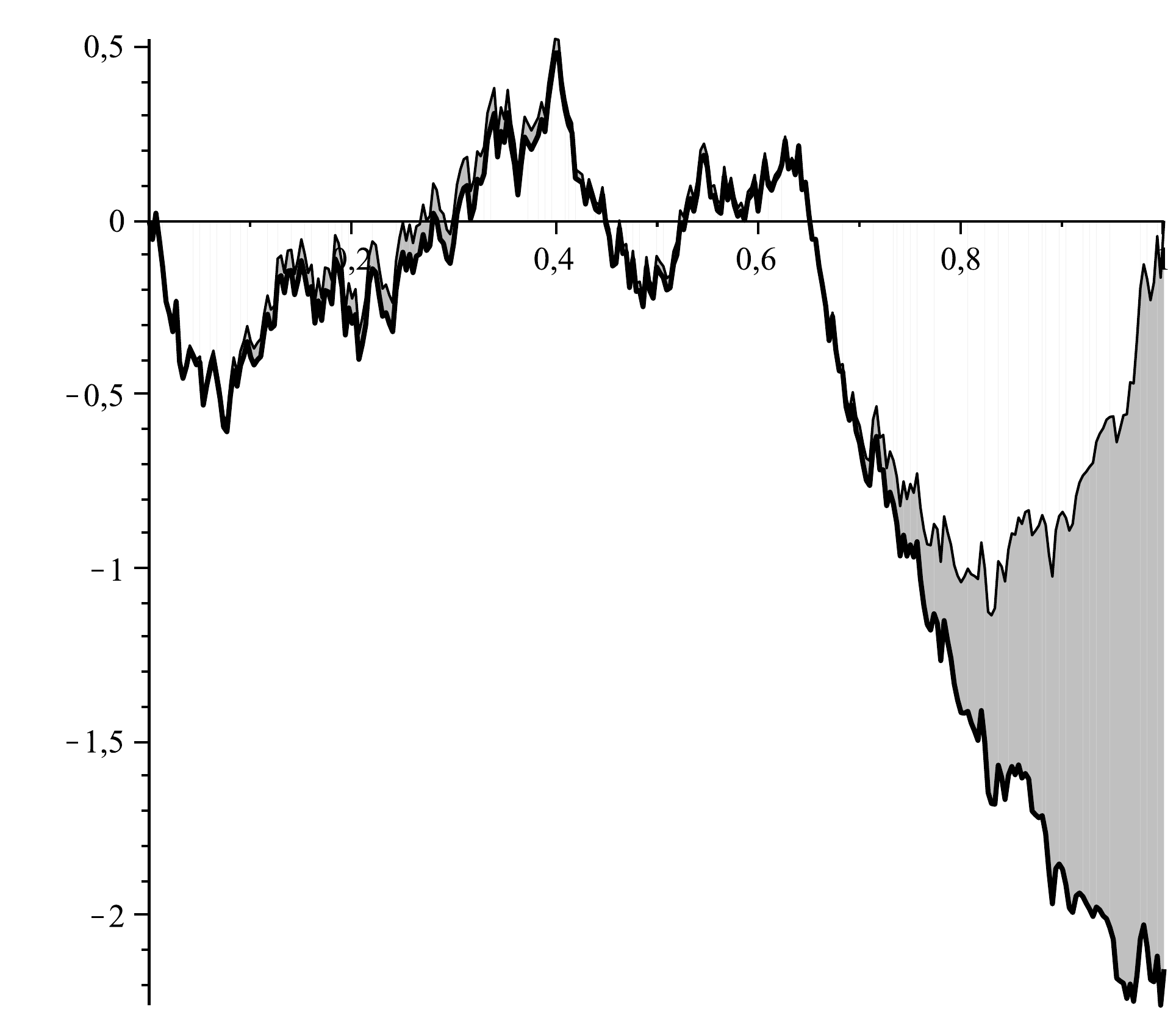}
\includegraphics[scale=0.25]{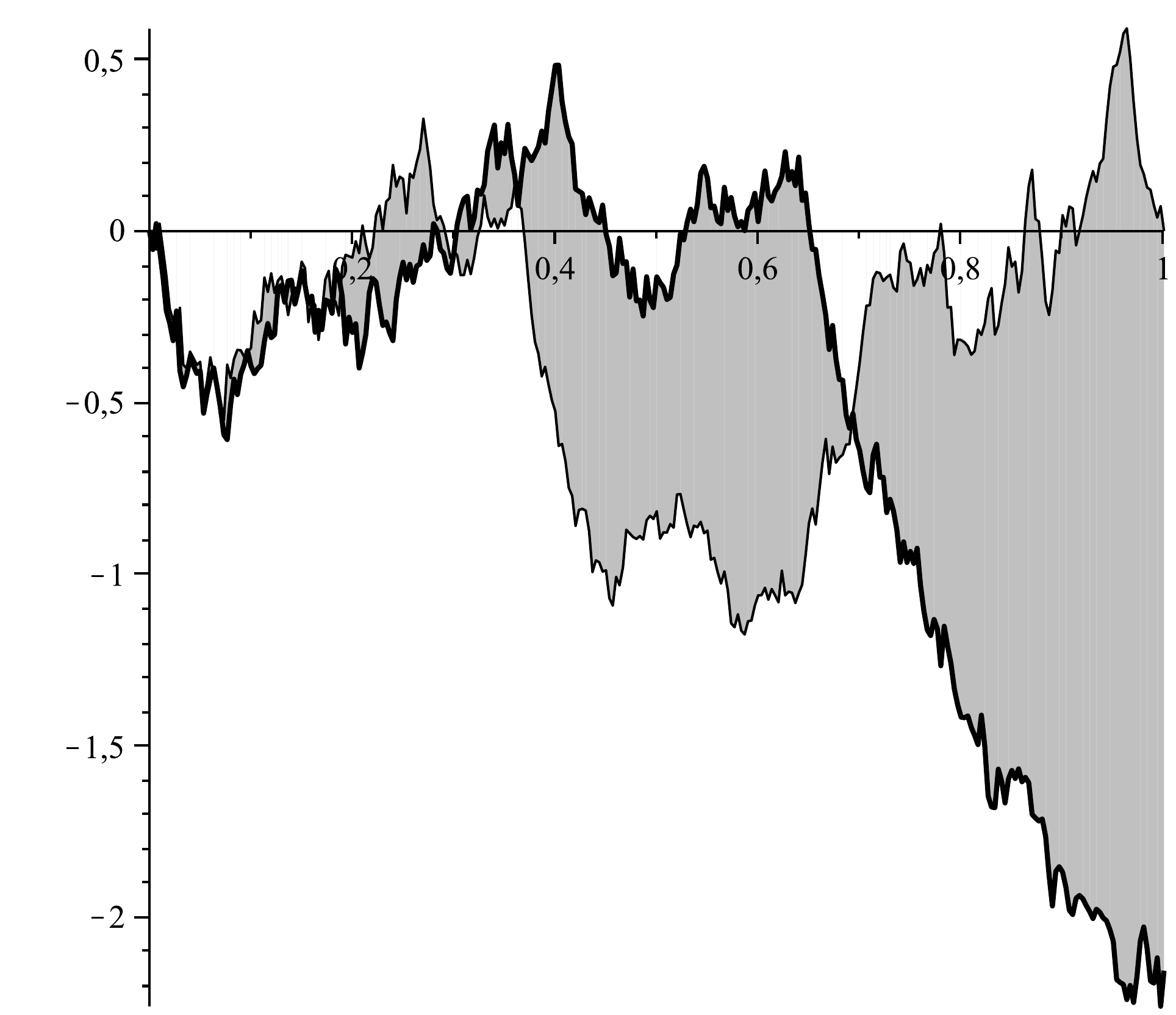}
\includegraphics[scale=0.25]{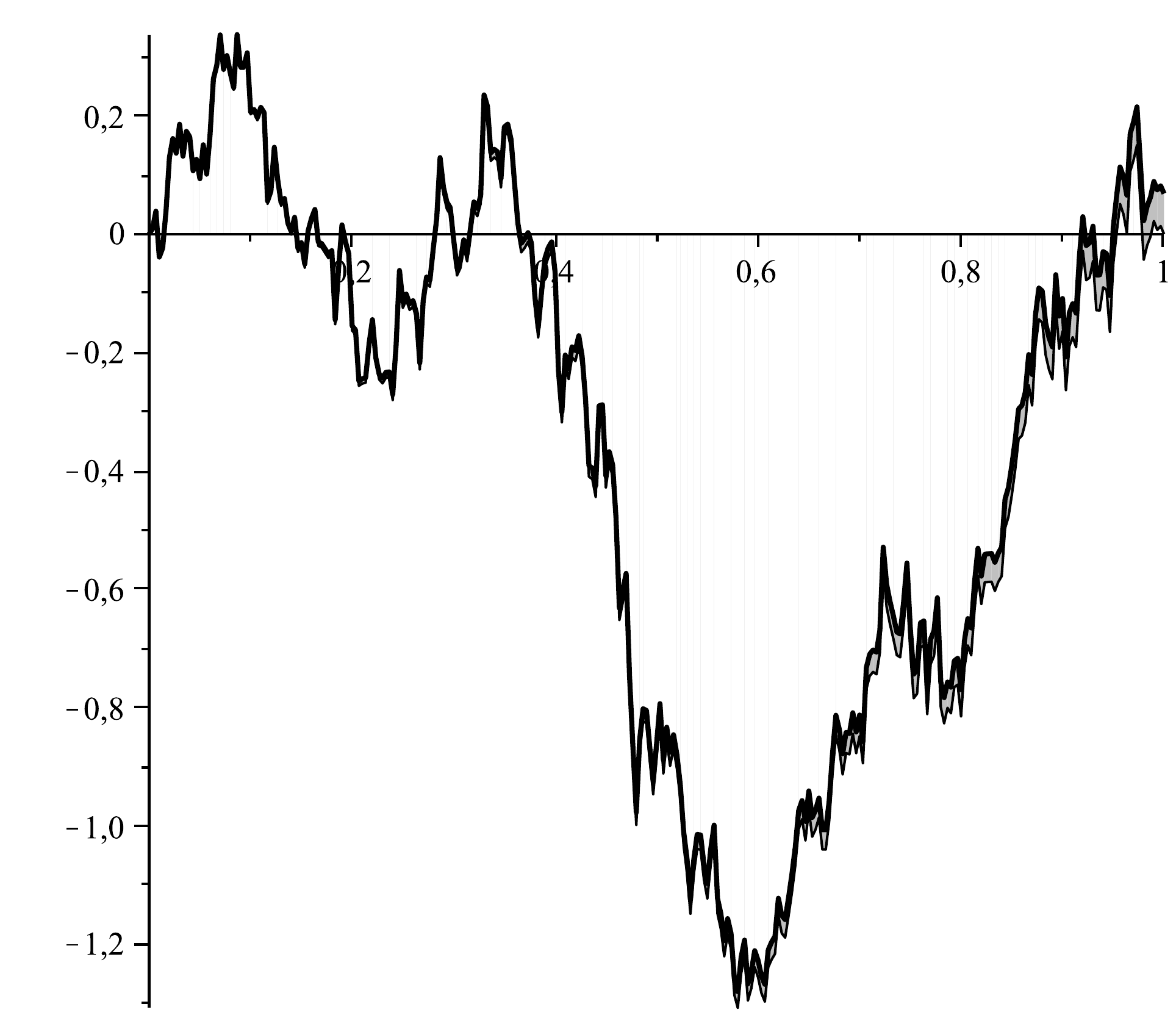}
\includegraphics[scale=0.25]{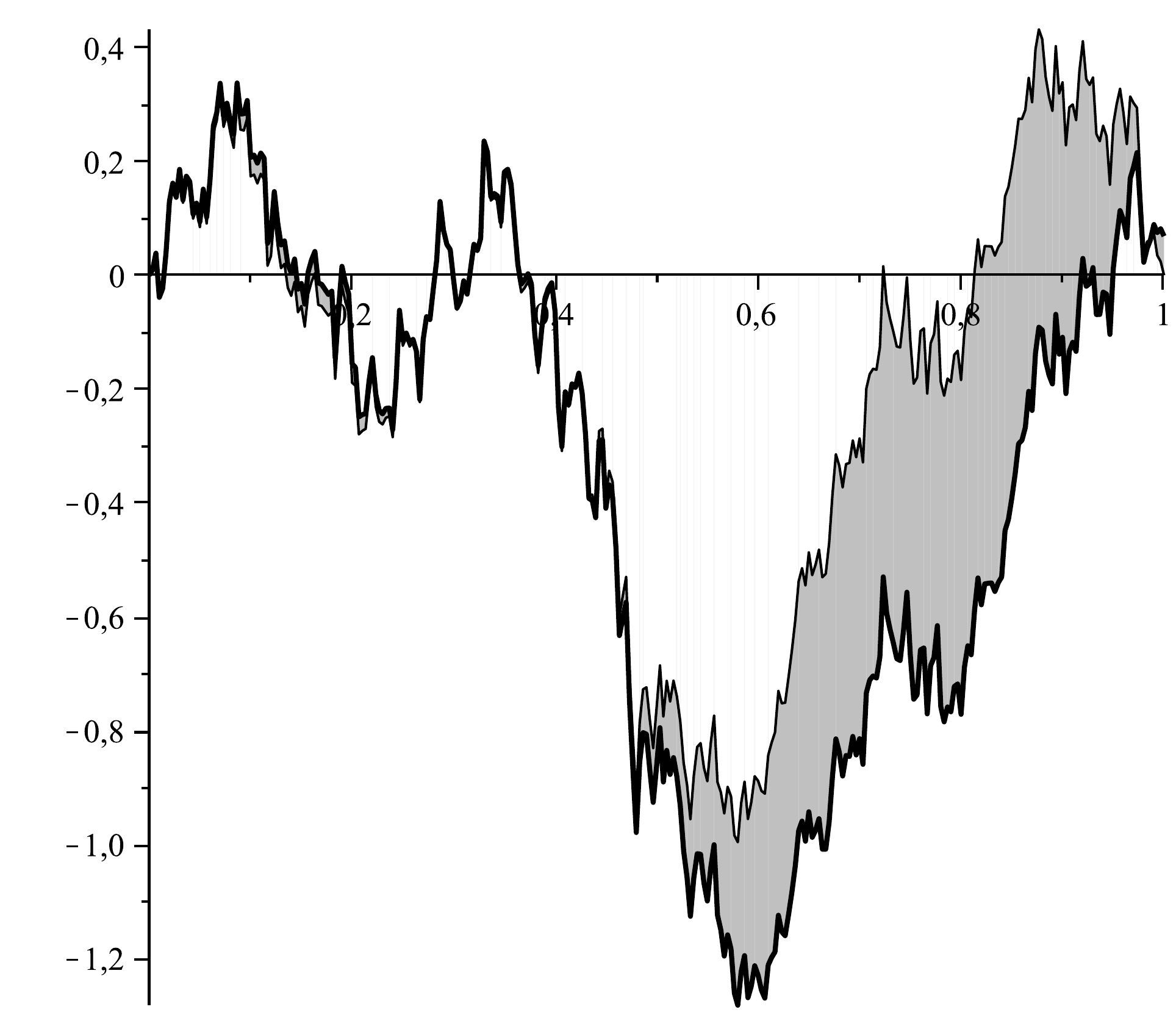}
\includegraphics[scale=0.25]{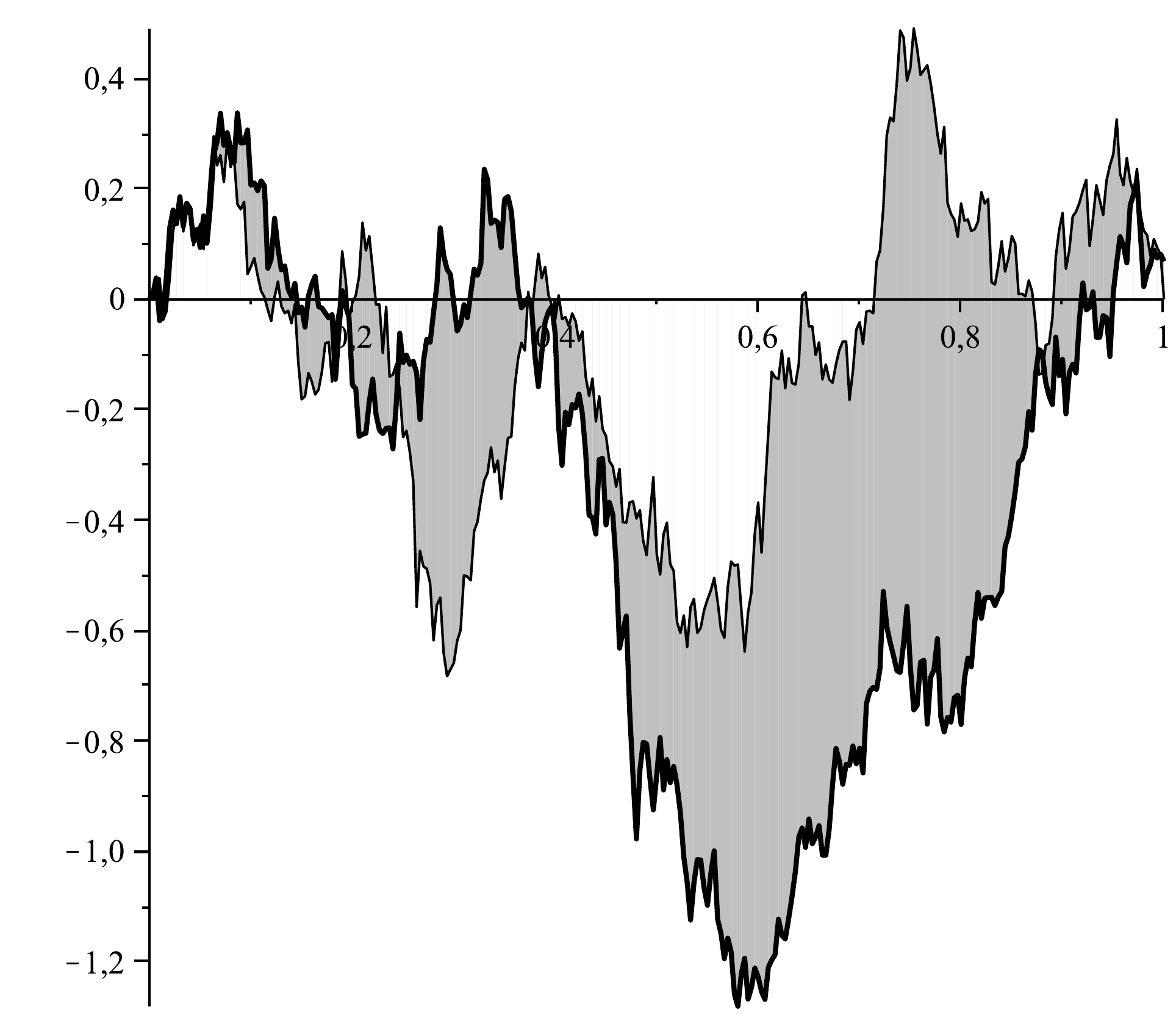}
\caption{\small Two sample paths of the Ornstein-Uhlenbeck process with $\sigma=1$ and $q=-1$
 (upper row, thick line), respectively $q=2$ (lower row, thick line) and their deviations from
 the anticipative version (left column), the integral representation (middle column),
 and the space-time transform (right column) of the Ornstein-Uhlenbeck bridge from $0$ to $0$
 over the time-interval $[0,1]$.}
\label{Figure4}
\end{figure}
Note that in general the deviation from the space-time transform bridge version is hard to compare
 with the other deviations, since $(U_{t}^{\rm st})_{t\in[t^\ast,T)}$ depends on the non-visible part
 $(U^{a}_t)_{t\in[T,\infty)}$ of the Ornstein-Uhlenbeck process, where $t^*\in(0,T)$ defined as follows.
Due to the strict monotonicity of $\kappa_{T}^\ast$ and $\lim_{t\uparrow T}\kappa_{T}^\ast(t)=\infty$ there
 is a unique $t^\ast\in(0,T)$ such that $\kappa_{T}^\ast(t^\ast)=T$, \
 see the analysis of the time-transform $\kappa_{T}^\ast$ in Lemma \ref{lemma_transform}.
From simulation studies we also get the above typical behavior is again reversed in case the
 endpoint $U^{a}_T$ of the Ornstein-Uhlenbeck sample path is close to the prescribed endpoint $b$ of
  its bridge, namely, for such a sample path of the Ornstein-Uhlenbeck process the deviation from
 its anticipative bridge version is smaller than from its integral representation of the bridge,
 see the lower row of Figure \ref{Figure4}.

Our aim is again to give quantitative answers to this qualitative behavior observed from simulation studies
  by studying the path deviations on $[0,T)$:
\begin{align}
& U^{a}_{t}-U_{t}^{\rm av}
   =(a\,\ee^{qT}-b)\,\frac{\sinh(qt)}{\sinh(qT)}+\frac{\sinh(qt)}{\sinh(qT)}\,U^{0}_{T},\nonumber\\
& U^{a}_{t}-U_{t}^{\rm ir}
    =(a\,\ee^{qT}-b)\,\frac{\sinh(qt)}{\sinh(qT)}
       +\sigma\int_{0}^t\left(\ee^{q(t-s)}-\frac{\sinh(q(T-t))}{\sinh(q(T-s))}\right)\,\dd W_{s},\label{oudevs}\\
& U^{a}_{t}-U_{t}^{\rm st}
    =(a\,\ee^{qT}-b)\,\frac{\sinh(qt)}{\sinh(qT)}
     +\left(U^{0}_{t}-\sigma\,\ee^{qt}\,\frac{\kappa(T)-\kappa(t)}{\kappa(T)}
          \,W_{\frac{\kappa(t)\kappa(T)}{\kappa(T)-\kappa(t)}}\right)\nonumber.
\end{align}
Note that all path deviations depend only on the transformed difference $(a\,\ee^{qT}-b)$
 of starting and endpoint of the bridge.
Hence in the sequel without loss of generality we can and will assume that $a=0$.
For simplicity we will concentrate on calculating the Gauss distributions of path deviations and
 to compare the expected quadratic path deviations only.

\subsection{Gauss distribution of path deviations and expected quadratic path deviations}

First we prove that path deviations have Gauss distribution.

\begin{prop}\label{prop7}
Let \ $(U_{t}^{\rm br})_{t\in[0,T]}$ \ be an Ornstein-Uhlenbeck bridge from \ $0$ \ to \ $b$ \
 over the time-interval \ $[0,T]$, \ where \ $b\in\RR$.
Then for all \ $t\in[0,T)$, the path deviation \ $U^{0}_t - U_{t}^{\rm br}$ \ is normally distributed
 with mean
 \[
    \Exp(U^{0}_t-U_{t}^{\rm br})=-b\,\frac{\sinh(qt)}{\sinh(qT)},
 \]
 and with variance
 \begin{align}\label{help_cond_gauss_ou1}
  & \Var(U^{0}_t-U_{t}^{\rm av})
      = \sigma^2\frac{\ee^{qT}}{q}\,\frac{\sinh^2(qt)}{\sinh(qT)},\\[1mm] \label{help_cond_gauss_ou2}
  & \Var(U^{0}_t-U_{t}^{\rm ir})
    = \frac{\sigma^2}{q}\left(\sinh(qt)\left(\ee^{qt}+\frac{\sinh(q(T-t))}{\sinh(qT)}\right)\right.\\\nonumber
  & \phantom{\Var(U^{0}_t-U_{t}^{\rm ir})=\sigma^2\,}\quad\left.-2\ee^{-q(T-t)}\sinh(q(T-t))
      \left(qt+\log\frac{\sinh(qT)}{\sinh(q(T-t))}\right)\right),\\[2mm]\label{help_cond_gauss_ou3}
  & \Var(U^{0}_t-U_{t}^{\rm st})
    = \frac{\sigma^2}{q}\left(\sinh(qt)\left(\ee^{qt}+\frac{\sinh(q(T-t))}{\sinh(qT)}\right)\right.\\\nonumber
  &\phantom{\Var(U^{0}_t-U_{t}^{\rm st}) = \frac{\sigma^2}{q}\;\;\;\;}
     \left.+ 2(1-\ee^{qt})\,\frac{\sinh(q(T-t))}{\sinh(qT)}\right).
 \end{align}
\end{prop}

\begin{proof}
With $a=0$, by \eqref{oudevs}, for every $0\leq t<T$ the path deviation $U^{0}_t-U_{t}^{\rm br}$ is
 normally distributed with mean $\Exp(U^{0}_t-U_{t}^{\rm br})=-b\,\frac{\sinh(qt)}{\sinh(qT)}$
 and with variance
$$
 \Var(U^{0}_t-U_{t}^{\rm av})
   = \frac{\sinh^2(qt)}{\sinh^2(qT)}\,\Var(U^{0}_T)=\sigma^2\frac{\ee^{qT}}{q}\,\frac{\sinh^2(qt)}{\sinh(qT)},
$$
and
\begin{align*}
 \Var(U^{0}_t-U_{t}^{\rm ir})
  &  = \sigma^2\int_{0}^t\left(\ee^{q(t-s)}-\frac{\sinh(q(T-t))}{\sinh(q(T-s))}\right)^2\,\dd s\\
  &  = \sigma^2\int_{0}^t\left[\ee^{2q(t-s)}-2\ee^{q(t-s)}\,\frac{\sinh(q(T-t))}{\sinh(q(T-s))}
        +\frac{\sinh^2(q(T-t))}{\sinh^2(q(T-s))}\right]\,\dd s\\
 & =\sigma^2\left(\frac1{2q}\,(\ee^{2qt}-1)
  -2\int_{0}^t\frac{\ee^{q(T-s)}-\ee^{-q(T-2t+s)}}{\ee^{q(T-s)}-\ee^{-q(T-s)}}\,\dd s\right.\\
 & \phantom{=\sigma^2}\quad\left.+\sinh^2(q(T-t))\int_{0}^t\frac1{\sinh^2(q(T-s))}\,\dd s\right)\\
 & =\sigma^2\left(\frac1{2q}\,(\ee^{2qt}-1)
    -\frac{2(1-\ee^{-2q(T-t)})}{q}\int_{\ee^{q(T-t)}}^{\ee^{qT}}\frac{v}{v^2-1}\,\dd v\right.\\
 & \phantom{=\sigma^2}\quad+\sinh^2(q(T-t))\,\frac1q\,
         \left(\frac{\cosh(q(T-t))}{\sinh(q(T-t))}-\frac{\cosh(qT)}{\sinh(qT)}\right)\Bigg)\\
 & =\sigma^2\left(\frac1{2q}\,(\ee^{2qt}-1)-\frac{(1-\ee^{-2q(T-t)})}{q}\log\frac{\ee^{2qT}-1}{\ee^{2q(T-t)}-1}\right.\\
 & \phantom{=\sigma^2}\quad\left.+\sinh^2(q(T-t))\,\frac1q\,\frac{\sinh(qt)}{\sinh(q(T-t))\sinh(qT)}\right),
\end{align*}
 which yields \eqref{help_cond_gauss_ou2}.
Using Lemma \ref{lemma_transform} we get
\begin{align*}
 \Var(U^{0}_t-U_{t}^{\rm st})
  & = \Var\left(U^{0}_t-\sigma\,\ee^{qt}
      \,\frac{\kappa(T)-\kappa(t)}{\kappa(T)}\,W_{\frac{\kappa(t)\kappa(T)}{\kappa(T)-\kappa(t)}}\right)\\
& =\Var(U^{0}_t)+\sigma^2\,
    \ee^{2qt}\left(\frac{\kappa(T)-\kappa(t)}{\kappa(T)}\right)^2\Var\left(W_{\frac{\kappa(t)\kappa(T)}{\kappa(T)-\kappa(t)}}\right)\\
& \phantom{=}-2\sigma\,\ee^{qt}\,
  \frac{\kappa(T)-\kappa(t)}{\kappa(T)}\,\Cov\left(U^{0}_t,W_{\frac{\kappa(t)\kappa(T)}{\kappa(T)-\kappa(t)}}\right)\\
&  =\sigma^2\frac{\ee^{qt}}{q}\,\sinh(qt)+\sigma^2\,\ee^{2qt}\kappa(t)\,\frac{\kappa(T)-\kappa(t)}{\kappa(T)}\\
& \phantom{=}-2\sigma^2\,\ee^{qt}\,
   \frac{\kappa(T)-\kappa(t)}{\kappa(T)}\int_{0}^{\min\{t,\kappa_{T}^\ast(t)\}}\ee^{q(t-s)}\,\dd s\\
& =\sigma^2\frac{\ee^{qt}}{q}
  \left(\sinh(qt)+\ee^{qt}\,\frac{1-\ee^{-2qt}}{2}\,\frac{\ee^{-2qt}-\ee^{-2qT}}{1-\ee^{-2qT}}\right.\\
& \phantom{=\sigma^2\frac{\ee^{qt}}{q}}\quad\left.-2\,\frac{\ee^{-2qt}-\ee^{-2qT}}{1-\ee^{-2qT}}\,(\ee^{qt}-1)\right),
\end{align*}
 which yields \eqref{help_cond_gauss_ou3}.
\end{proof}

Note that in the proof of Proposition \ref{prop8} below we will give different representations of the
 variances \ $\Var(U^{0}_t-U_{t}^{\rm br})$ \ calculated in Proposition \ref{prop7}.

Next we compare the second moments \ $\Exp((U^{0}_t-U_t^{\rm br})^2)$ \ of the path deviations
 \ $U^{0}_t-U_t^{\rm br}$.
In view of \eqref{quaddist} we have to compare the variances of path deviations for different versions of
 the Ornstein-Uhlenbeck bridge, since the mean function of path deviation is the same for all versions.

\begin{prop}\label{prop8}
Let \ $(U_{t}^{\rm br})_{t\in[0,T]}$ \ be an Ornstein-Uhlenbeck bridge from \ $0$ \ to \ $b$ \
 over the time-interval \ $[0,T]$, \ where \ $b\in\RR$.
Then for all \ $t\in(0,T)$, \ we have
\begin{equation}\label{ouquaddist}
 \Exp\big((U^{0}_t-U_{t}^{\rm ir})^2\big)
  <\begin{cases}
     \Exp\big((U^{0}_t-U_{t}^{\rm st})^2\big)<\Exp\big((U^{0}_t-U_{t}^{\rm av})^2\big),&\text{ if } \; q>0,\\[1mm]
     \Exp\big((U^{0}_t-U_{t}^{\rm av})^2\big)<\Exp\big((U^{0}_t-U_{t}^{\rm st})^2\big),&\text{ if } \;q<0.
   \end{cases}
\end{equation}
\end{prop}

\begin{proof}
We first give different representations of \ $\Var(U^{0}_t-U_{t}^{\rm ir})$ \ and
 \ $\Var(U^{0}_t-U_{t}^{\rm st})$ \ calculated in Proposition \ref{prop7} that are
 more suitable for comparison.
By Proposition \ref{prop7}, we have
 \begin{align}\label{help_exp_quad_ou_path_dev1}
  \begin{split}
  \Var&(U^{0}_t-U_{t}^{\rm ir}) \\
   &=\frac{\sigma^2}{q}\left[
    2\sinh(q(T-t))\left(\frac{\sinh(qt)}{\sinh(qT)}-\ee^{-q(T-t)}
                      \left(qt+\log\frac{\sinh(qT)}{\sinh(q(T-t))}\right)\right)\right. \\
   &\phantom{=\frac{\sigma^2}{q}\;\left[\right.}\left.
     + \sinh(qt)\left(\ee^{qt}-\frac{\sinh(q(T-t))}{\sinh(qT)}\right)
    \right]\\
   & = 2\,\frac{\sigma^2}{q}\,\sinh(q(T-t))
      \left(\frac{\sinh(qt)}{\sinh(qT)}-\ee^{-q(T-t)}\left(qt+\log\frac{\sinh(qT)}{\sinh(q(T-t))}\right)\right)\\
   &\phantom{ =\;}
      + \ee^{qT}\,\frac{\sigma^2}{q}\,\frac{\sinh^2(qt)}{\sinh(qT)},
  \end{split}
 \end{align}
and
\begin{align}\label{help_exp_quad_ou_path_dev2}
 \begin{split}
  \Var(U^{0}_t-U_{t}^{\rm st})
   & =\frac{\sigma^2}{q}\left(\sinh(qt)\left(\ee^{qt}-\frac{\sinh(q(T-t))}{\sinh(qT)}\right)\right.\\
   & \phantom{=\sigma^2}\quad\left.+2\,\frac{\sinh(q(T-t))}{\sinh(qT)}\,(\sinh(qt)+1-\ee^{qt})\right)\\
   & = \ee^{qT}\,\frac{\sigma^2}{q}\,\frac{\sinh^2(qt)}{\sinh(qT)}+2\,\frac{\sigma^2}{q}
      \,\frac{\sinh(q(T-t))(1-\cosh(qt))}{\sinh(qT)}.
 \end{split}
\end{align}
The advantage of this new representation is that now the variances include the term
 $\ee^{qT}\,\frac{\sigma^2}{q}\,\frac{\sinh^2(qt)}{\sinh(qT)}$ for all versions of path deviations.

For the comparison $\Exp(U^{0}_t-U_{t}^{\rm st})^2$ with $\Exp(U^{0}_t-U_{t}^{\rm av})^2$ we
 consider the continuous function $h_{q}$ on $[0,T]$ defined by
 $$
   h_{q}(t):=2\,\frac{\sigma^2}{q}\,\frac{\sinh(q(T-t))(1-\cosh(qt))}{\sinh(qT)},
      \qquad t\in[0,T].
 $$
Clearly, $h_{q}(t)=0$ if and only if $t\in\{0,T\}$ and further for all $0<t<T$ we have $h_{q}(t)<0$ if $q>0$ and $h_{q}(t)>0$ if $q<0$. In view of \eqref{quaddist} we get
 \begin{equation}\label{oudistpart}
 \Exp\big((U^{0}_t-U_{t}^{\rm st})^2\big)
   \begin{cases}
     <\Exp\big((U^{0}_t-U_{t}^{\rm av})^2\big), & \text{ if } q>0,\\
     >\Exp\big((U^{0}_t-U_{t}^{\rm av})^2\big), & \text{ if } q<0.
   \end{cases}
\end{equation}
For the other comparisons, we show that
 $$
 \frac{\sinh(qt)}{\sinh(qT)}-\ee^{-q(T-t)}\left(qt+\log\frac{\sinh(qT)}{\sinh(q(T-t))}\right)
   <\begin{cases}
      0 &\text{ if }q<0,\\
      \displaystyle\frac{1-\cosh(qt)}{\sinh(qT)}&\text{ if }q>0.
     \end{cases}
$$
Using that
 \[
   \left\vert \frac{\sinh(q(T-t))}{\sinh(qT)}-1\right\vert<1,\qquad t\in(0,T),
 \]
 by $\log(1+x)\leq x$, $\vert x\vert<1$, we have for all $0<t<T$,
\begin{align*}
& \frac{\sinh(qt)}{\sinh(qT)}-\ee^{-q(T-t)}\left(qt+\log\frac{\sinh(qT)}{\sinh(q(T-t))}\right)\\
& \quad=\frac{\sinh(qt)}{\sinh(qT)}+\ee^{-q(T-t)}\left(\log\frac{\sinh(q(T-t))}{\sinh(qT)}-qt\right)\\
& \quad\leq\frac{\sinh(qt)}{\sinh(qT)}+\ee^{-q(T-t)}\,\frac{\sinh(q(T-t))-(1+qt)\sinh(qT)}{\sinh(qT)}\\
& \quad=\frac1{2\sinh(qT)}\left(\ee^{qt}-\ee^{-qt}+1-\ee^{-2q(T-t)}-(1+qt)\left(\ee^{qt}-\ee^{-2qT+qt}\right)\right)\\
& \quad=:\frac1{2\sinh(qT)}\,g_{q}(t).
\end{align*}
For $q<0$ it is enough to show that $g_{q}(t)>0$ for all $0<t<T$. Now
\begin{align*}
g_{q}(t) & =\big(\ee^{-2qT+qt}-\ee^{-2q(T-t)}\big)+\big(1-\ee^{-qt}\big)-qt\big(\ee^{qt}-\ee^{-2qT+qt}\big)\\
& =\ee^{-2q(T-t)}\big(\ee^{-qt}-1\big)+\big(1-\ee^{-qt}\big)-qt\,\ee^{qt}\big(1-\ee^{-2qT}\big)\\
& =\big(\ee^{-2q(T-t)}-1\big)\big(\ee^{-qt}-1\big)+qt\,\ee^{qt}\big(\ee^{-2qT}-1\big)\\
& \leq\big(\ee^{-2qT}-1\big)\big(\ee^{-qt}-1\big)+qt\,\ee^{qt}\big(\ee^{-2qT}-1\big)\\
& =\big(\ee^{-2qT}-1\big)\big(\ee^{-qt}-1+qt\big)+qt\big(\ee^{qt}-1\big)\big(\ee^{-2qT}-1\big),
\end{align*}
which is obviously positive for $q<0$ and $0<t<T$. For $q>0$ we have to show that
 $g_{q}(t)<2-\ee^{qt}-\ee^{-qt}$ for all $0<t<T$. Now
\begin{align*}
2-\ee^{qt}-\ee^{-qt}-g_{q}(t) & =1-\ee^{qt}+\ee^{-2q(T-t)}-\ee^{-2qT+qt}+qt\,\ee^{qt}\big(1-\ee^{-2qT}\big)\\
& =:\tilde g_{q}(t)
\end{align*}
for which $\tilde g_{q}(0)=0$ holds and we have
\begin{align*}
\tilde g'_{q}(t) & =-q\,\ee^{qt}+2q\,\ee^{-2q(T-t)}-q\,\ee^{-2qT+qt}\\
& \quad\,+q\,\ee^{qt}\big(1-\ee^{-2qT}\big)+q^2t\,\ee^{qt}\big(1-\ee^{-2qT}\big)\\
& =2q\,\ee^{-2qT+qt}\left(\ee^{qt}-1\right)+q^2t\,\ee^{qt}\big(1-\ee^{-2qT}\big)>0
\end{align*}
for $q>0$ and $0<t<T$, which completes the proof.
Hence, by \eqref{quaddist}, \eqref{help_cond_gauss_ou1}, \eqref{help_exp_quad_ou_path_dev1},
 \eqref{help_exp_quad_ou_path_dev2} and \eqref{oudistpart}, we get \eqref{ouquaddist}.
\end{proof}

Moreover, by \eqref{ouquaddist}, the expected quadratic path deviations satisfy the following
 inequalities:
$$
 \int_{0}^T\Exp\big((U^{0}_t-U_{t}^{\rm ir})^2\big)\,\dd t
  <
  \begin{cases}
  \displaystyle\int_{0}^T\Exp\big((U^{0}_t-U_{t}^{\rm st})^2\big)\,\dd t
         <\int_{0}^T\Exp\big((U^{0}_t-U_{t}^{\rm av})^2\big)\,\dd t&\text{ if }q>0,\\
  \displaystyle\int_{0}^T\Exp\big((U^{0}_t-U_{t}^{\rm av})^2\big)\,\dd t
         <\int_{0}^T \Exp\big((U^{0}_t-U_{t}^{\rm st})^2\big)\,\dd t&\text{ if }q<0.
 \end{cases}$$

In the next theorem we get more explicit representations of the expected quadratic path deviations.

\begin{thm}\label{prop10}
Let \ $(U_{t}^{\rm br})_{t\in[0,T]}$ \ be an Ornstein-Uhlenbeck bridge from \ $0$ \ to \ $b$ \
 over the time-interval \ $[0,T]$, \ where \ $b\in\RR$.
\ Then we have
 \begin{align}\nonumber
   & \Exp\left(\int_0^T(U^{0}_t-U_{t}^{\rm av})^2\,\dd t\right)
       = \frac{b^2}{4q}\cdot\frac{\sinh(2qT)-2qT}{\sinh^2(qT)}
         + \frac{\sigma^2\ee^{qT}}{4q^2}\cdot\frac{\sinh(2qT)-2qT}{\sinh(qT)},\\[2mm]\nonumber
   &   \Exp\left(\int_0^T(U^{0}_t-U_{t}^{\rm ir})^2\,\dd t\right)
       = \frac{b^2}{4q}\cdot\frac{\sinh(2qT)-2qT}{\sinh^2(qT)}
        + \frac{\sigma^2 \ee^{qT}}{4q^2}\cdot\frac{\sinh(2qT)-2qT}{\sinh(qT)}
        - \frac{\sigma^2}{q^2} \\\nonumber
   &\phantom{\Exp\left(\int_0^T(U^{0}_t-U_{t}^{\rm ir})^2\,\dd t\right) =}
        + \frac{T\sigma^2\cosh(qT)}{q\sinh(qT)}
        - \frac{\sigma^2T^2}{2}
        + \frac{\sigma^2T}{2q}
        - \frac{\sigma^2}{4q^2}
       + \frac{\sigma^2}{4q^2}\ee^{-2qT} \\\nonumber
   &\phantom{\Exp\left(\int_0^T(U^{0}_t-U_{t}^{\rm ir})^2\,\dd t\right) =}
      - \frac{\sigma^2}{q^2}
        \int_0^{qT}(1-\ee^{-2x})
        \log \frac{\sinh(qT)}{\sinh(x)}\, \dd x,\\[1mm]\nonumber
   &\Exp\left(\int_0^T(U^{0}_t-U_{t}^{\rm st})^2\,\dd t\right)
     = \frac{\sigma^2\ee^{qT}}{4q^2}\cdot\frac{\sinh(2qT)-2qT}{\sinh(qT)}
        + \frac{2\sigma^2}{q^2}\cdot\frac{\cosh(qT)-1}{\sinh(qT)} \\\nonumber
   &\phantom{\Exp\left(\int_0^T(U^{0}_t-U_{t}^{\rm st})^2\,\dd t\right) = }
        - \frac{\sigma^2}{q}\left(\frac{\sinh(2qT)}{2q}+T\right)
        + \frac{\sigma^2 \cosh(qT)}{2q^2\sinh(qT)}(\cosh(2qT)-1).
 \end{align}
\end{thm}

\begin{proof}
For the anticipative version, by Proposition \ref{prop7}, we get
 \begin{align*}
   &\Exp\left(\int_0^T(U^{0}_t-U_{t}^{\rm av})^2\,\dd t\right)
      = \int_0^T\Exp(U^{0}_t-U_{t}^{\rm av})^2\,\dd t\\
   &\qquad = \int_0^T \ee^{qT}\frac{\sigma^2}{q}\frac{\sinh^2(qt)}{\sinh(qT)}\,\dd t
       + \int_0^T b^2 \frac{\sinh^2(qt)}{\sinh^2(qT)}\,\dd t \\
   &\qquad = \frac{\ee^{qT}\sigma^2}{q\sinh(qT)} \int_0^T \sinh^2(qt)\,\dd t
       + \frac{b^2}{\sinh^2(qT)} \int_0^T \sinh^2(qt)\,\dd t \\
   &\qquad = \frac{\ee^{qT}\sigma^2}{2q\sinh(qT)}
       \left(\frac{\sinh(2qT)}{2q}-T\right)
      + \frac{b^2}{2\sinh^2(qT)} \left(\frac{\sinh(2qT)}{2q}-T\right).
 \end{align*}
For the integral representation, by Proposition \ref{prop7} and the previous calculations for
 the anticipative version, we get
 \begin{align*}
   &\Exp\left(\int_0^T(U^{0}_t-U_{t}^{\rm ir})^2\,\dd t\right)
      = \int_0^T\Exp(U^{0}_t-U_{t}^{\rm ir})^2\,\dd t \\
   &\phantom{\qquad} = \frac{b^2}{4q}\,\frac{\sinh(2qT)-2qT}{\sinh^2(qT)}
            + \frac{\sigma^2 \ee^{qT}}{4q^2}\frac{\sinh(2qT)-2qT}{\sinh(qT)} \\
   &\phantom{\qquad =\; }  - \frac{2\sigma^2}{q} \int_0^T \sinh(q(T-t)) \ee^{-q(T-t)}
            \left(qt + \log \frac{\sinh(qT)}{\sinh(q(T-t))}\right)\, \dd t\\
   &\phantom{\qquad = \;} + \frac{2\sigma^2}{q\sinh(qT)} \int_0^T \sinh(q(T-t))\sinh(qt)\, \dd t.
 \end{align*}
Here
 \begin{align*}
   \int\sinh(q(T-t))\sinh(qt)&\, \dd t
     = \int\Big(\sinh(qT)\cosh(qt) - \cosh(qT)\sinh(qt)\Big)\sinh(qt)\, \dd t \\
    & = \sinh(qT) \int \cosh(qt)\sinh(qt)\, \dd t
       - \cosh(qT) \int \sinh^2(qt)\, \dd t \\
    & = \frac{\sinh(qT)}{2} \int \sinh(2qt)\, \dd t
        - \frac{\cosh(qT)}{2}\left(\frac{\sinh(2qt)}{2q}-t\right) \\
    & = \frac{\sinh(qT)\cosh(2qt)}{4q}
       - \frac{\cosh(qT)\sinh(2qt)}{4q}
       +\frac{t\cosh(qT)}{2} \\
   & = \frac{\sinh(q(T-2t))}{4q} + \frac{t\cosh(qT)}{2},
 \end{align*}
 and hence
 \begin{align*}
    \frac{2\sigma^2}{q\sinh(qT)}
       &\int_0^T \sinh(q(T-t))\sinh(qt)\, \dd t\\
     &= \frac{2\sigma^2}{q\sinh(qT)}
        \left( -\frac{\sinh(qT)}{4q} + \frac{T\cosh(qT)}{2}
               - \frac{\sinh(qT)}{4q}\right) \\
    & = -\frac{\sigma^2}{q^2}
       +\frac{T\sigma^2\cosh(qT)}{q\sinh(qT)}.
 \end{align*}
We also have, by partial integration,
 \begin{align*}
   \int\sinh(q(T-t))\ee^{-q(T-t)}qt\, \dd t
   & = \int\frac{1-\ee^{-2q(T-t)}}{2}qt\, \dd t
     = \frac{q}{2}\left(\frac{t^2}{2}
       - \int t\ee^{-2q(T-t)}\, \dd t \right)\\
   & = \frac{qt^2}{4}
       - \frac{q}{2}\left(\frac{t\ee^{-2q(T-t)}}{2q}
                  - \int\frac{\ee^{-2q(T-t)}}{2q} \, \dd t\right) \\
   & = \frac{qt^2}{4} - \frac{t\ee^{-2q(T-t)}}{4}
      + \frac{\ee^{-2q(T-t)}}{8q},
 \end{align*}
 and hence
 \[
   -\frac{2\sigma^2}{q}\int_0^T\sinh(q(T-t))\ee^{-q(T-t)}qt\, \dd t
      = -\frac{2\sigma^2}{q}\left(\frac{qT^2}{4}-\frac{T}{4} +\frac{1}{8q}
           - \frac{1}{8q}\ee^{-2qT}\right).
 \]
Moreover, by the change of variables $q(T-t)=x$, we get
 \begin{align*}
   -\frac{2\sigma^2}{q}
     &  \int_0^T \sinh(q(T-t)) \ee^{-q(T-t)}
               \log \frac{\sinh(qT)}{\sinh(q(T-t))}\, \dd t\\
     & = -\frac{2\sigma^2}{q}\int_0^{qT} \frac{1-\ee^{-2x}}{2}
          \log \frac{\sinh(qT)}{\sinh(x)}\,\dd x,
 \end{align*}
 and then we get the formula for \ $\Exp\left(\int_0^T(U^{0}_t-U_{t}^{\rm ir})^2\,\dd t\right)$.
\ We note that we are unable to solve the integral
 \ $\int_0^{qT}(1-\ee^{-2x}) \log \frac{\sinh(qT)}{\sinh(x)}\, \dd x.$

Finally, for the space-time transform, by Proposition \ref{prop7} and the previous calculations
 for the anticipative version, we get
 \begin{align*}
   &\Exp\left(\int_0^T(U^{0}_t-U_{t}^{\rm st})^2\,\dd t\right)
      = \int_0^T\Exp(U^{0}_t-U_{t}^{\rm st})^2\,\dd t \\
   &  = \frac{\sigma^2\ee^{qT}}{4q^2}\cdot\frac{\sinh(2qT)-2qT}{\sinh(qT)}
         + \frac{2\sigma^2}{q\sinh(qT)}
            \int_0^T \sinh(q(T-t))(1-\cosh(qt))\,\dd t \\
   & = \frac{2\sigma^2}{q\sinh(qT)}
          \left(-\frac{1}{q}(1-\cosh(qT))
                -\int_0^T\sinh(q(T-t))\cosh(qt)\, \dd t
          \right) \\
   &\phantom{=\;} + \frac{\sigma^2\ee^{qT}}{4q^2}\cdot\frac{\sinh(2qT)-2qT}{\sinh(qT)}.
 \end{align*}
Here
 \begin{align*}
    \int_0^T&\sinh(q(T-t))\cosh(qt)\, \dd t \\
      &  = \sinh(qT) \int_0^T\cosh^2(qt)\, d t - \cosh(qT) \int_0^T\sinh(qt)\cosh(qt)\, \dd t \\
      & = \frac{\sinh(qT)}{2}\left(\frac{\sinh(2qT)}{2q}+T\right)
          - \frac{\cosh(qT)}{4q}(\cosh(2qT)-1),
 \end{align*}
 and hence we get the formula for \ $\Exp\left(\int_0^T(U^{0}_t-U_{t}^{\rm st})^2\,\dd t\right)$.
\end{proof}

We note that the formulas \ $\Exp\left(\int_0^T(U^{0}_t-U_{t}^{\rm br})^2\,\dd t\right)$ \
 are even harder to compare with each other than the variances in Proposition \ref{prop7} with each other.
It might also be possible to calculate the Gauss conditional distribution of path deviations given
 $U^{0}_t=d$ using Theorem 2 and Problem 5 in Chapter II, \S13 of Shiryaev \cite{Shi},
 and to calculate corresponding formulas for conditional quadratic path deviations.
But even if these formulas are present, the conditional quadratic path deviations will be hard to
 compare, since they will depend on the four parameters $q,b,d,T$ and possibly also on $\sigma$.
We renounce to give these explicit and likewise very long calculations.

\section{Appendix}

The following lemma yields almost sure (left) continuity at \ $t=T$ \
 of the integral representation of an Ornstein-Uhlenbeck bridge.

\begin{lemma}
Let \ $T\in(0,\infty)$ \ be fixed and let \ $(B_s)_{s\geq 0}$ \ be a one-dimensional standard Wiener process
 on a filtered probability space \ $(\Omega,\cA,(\cA_t)_{t\geq 0},\PP)$, \ where the filtration
 \  $(\cA)_{t\geq 0}$ \ is the usual augmentation of the natural filtration of the Wiener process
 \ $B$ \ (see, e.g., Karatzas and Shreve \cite[Section 5.2.A]{KarShr}).
The process \ $(Y_t)_{t\in[0,T]}$ \ defined by
 \[
    Y_t:=
         \begin{cases}
            \int_0^t\frac{\sinh(q(T-t))}{\sinh(q(T-s))}\,\dd B_s
               & \text{if \ $t\in[0,T)$,}\\
             0 & \text{if \ $t=T$,}
         \end{cases}
 \]
 is a centered Gauss process with almost sure continuous paths.
\end{lemma}

\begin{proof}
By Bauer \cite[Lemma 48.2]{Bau}, \ $(Y_t)_{t\in[0,T]}$ \ is a centered Gauss process.
To prove almost sure continuity, we follow the method of the proof of Lemma 5.6.9 in
 Karatzas and Shreve \cite{KarShr}.
For all \ $t\in[0,T)$, \ let
 \[
   M_t:=\int_0^t\frac{1}{\sinh(q(T-s))}\, \dd B_s.
 \]
Then \ $(M_t)_{t\in[0,T)}$ \ is a continuous, square-integrable martingale with respect to
 the filtration \ $(\cA_t)_{t\in[0,T)}$ \ and with quadratic variation
 \begin{align*}
   \langle M\rangle_t
     :=\int_0^t\frac{1}{\sinh^2(q(T-s))}\,\dd s
      =\frac{1}{q}(\coth(q(T-t))-\coth(qT)),\qquad t\in[0,T).
 \end{align*}
Then \ $\lim_{t\uparrow T}\langle M\rangle_t=\infty$.
By a strong law of large numbers for continuous local martingales, we get
 \[
    \PP\left(\lim_{t\uparrow T}\frac{M_t}{\langle M\rangle_t}=0\right)=1,
 \]
 see, e.g., L\'epingle \cite[Theoreme 1]{Lep} or $3^\circ)$ in Exercise 1.16 in Chapter V
 in Revuz and Yor \cite{RevYor}.
(We note that the above mentioned citations are about continuous local martingales with
 time-interval \ $[0,\infty)$, \ but they are also valid for continuous local martingales with
 time-interval \ $[0,T)$, \ $T\in(0,\infty)$, \ with appropriate modifications in their conditions,
 for such a formulation, see, e.g., Barczy and Pap \cite[Theorem 3.2]{BarPap}.)
Then we have
 \[
    Y_t=\sinh(q(T-t))M_t = \sinh(q(T-t))\langle M\rangle_t \frac{M_t}{\langle M\rangle_t},
        \qquad t\in(0,T).
 \]
Here
 \begin{align*}
   \lim_{t\uparrow T} \sinh(q(T-t))\langle M\rangle_t
      = \lim_{t\uparrow T} \frac{1}{q}\left( \cosh(q(T-t)) - \sinh(q(T-t))\coth(qT)\right)
      = \frac{1}{q}.
 \end{align*}
Hence we conclude \ $\PP(\lim_{t\uparrow T} Y_t=0)=1$.
\end{proof}

\bibliographystyle{plain}

\end{document}